\documentclass[12pt]{amsart}
\usepackage{amsmath,amssymb,amsthm}
\usepackage{color}   
\usepackage{hyperref}
\usepackage{graphicx}
\usepackage{wrapfig} 

\usepackage{soul}

\usepackage[usenames,dvipsnames]{pstricks}
 \usepackage{epsfig}
 \usepackage{pst-grad} 
 \usepackage{pst-plot} 

\newtheorem{lem}{Lemma}[section]
\newtheorem{prop}{Proposition}[section]
\newtheorem{cor}{Corollary}[section]

\newtheorem{thm}{Theorem}[section]

\theoremstyle{definition}

\theoremstyle{remark}

\theoremstyle{remark}
\newtheorem{remark}{Remark}[section]
\numberwithin{equation}{section}

\newcommand{\N}{{\mathbb N}}

\newcommand{\R}{{\mathbb R}}

\definecolor{blu}{rgb}{0,0,1}

\title[Continuum of solutions]{Continuum of solutions for an elliptic problem with critical growth in the gradient}

\author[D. Arcoya]{David Arcoya}
\thanks{D. A. is supported by Ministerio de Econom\'{\i}a y Competitividad (Spain) MTM2012-31799 and Junta de Andaluc\'{\i}a FQM-116.}
\address{David Arcoya 
\newline\indent
Departamento de An\'alisis Matem\'atico, Universidad de Granada, 
\newline\indent
C/Severo Ochoa, 18071 Granada, Spain }
\email{darcoya@ugr.es}

\author[C. De Coster]{Colette De Coster }
\address{Colette de Coster
\newline\indent
Universit\'e de Valenciennes et du Hainaut Cambr\'esis
\newline\indent
LAMAV,  FR CNRS 2956, 
\newline\indent
Institut des Sciences et Techniques de Valenciennes
\newline\indent
F-59313  Valenciennes Cedex 9, France}
\email{Colette.DeCoster@univ-valenciennes.fr}

\author[L. Jeanjean]{Louis Jeanjean}
\address{Louis Jeanjean
\newline\indent
Laboratoire de Math\'ematiques (UMR 6623)
\newline\indent
Universit\'{e} de Franche-Comt\'{e}
\newline\indent
16, Route de Gray 25030 Besan\c{c}on Cedex, France}
\email{louis.jeanjean@univ-fcomte.fr}

\author[K. Tanaka]{Kazunaga Tanaka}
\address{Kazunaga Tanaka
\newline\indent
Department of Mathematics, 
\newline\indent
School of Science and Engineering
\newline\indent
Waseda University
\newline\indent
3-4-1 Ohkubo, Shijuku-ku, Tokyo 169-8555, Japan}
\email{kazunaga@waseda.jp}

\begin{document}


\begin{abstract}
We consider the boundary value problem
\begin{equation*}
- \Delta u = \lambda c(x)u+ \mu(x) |\nabla u|^2 + h(x), \quad u \in H^1_0(\Omega) \cap L^{\infty}(\Omega)
\eqno{(P_{\lambda})}
\end{equation*}
where $\Omega \subset \R^N, N \geq 3$ is a bounded domain with smooth boundary. It is assumed that $c\gneqq 0$, $c,h$ belong to $L^p(\Omega)$ for some  $p > N/2$ and that $\mu \in L^{\infty}(\Omega).$ We explicit a condition which guarantees the existence of a unique solution of $(P_{\lambda})$ when $\lambda <0$ and we show that these solutions belong to a continuum. The behaviour of the continuum depends in an essential way on the existence of a solution of $(P_0)$. It crosses the axis $\lambda =0$ if $(P_0)$ has a solution, otherwise if bifurcates from infinity at the left of the axis $\lambda =0$. Assuming that $(P_0)$ has a solution and strenghtening our assumptions to $\mu(x)\geq \mu_1>0$ and $h\gneqq 0$,
 we show that the continuum bifurcates from infinity on the right of the axis $\lambda =0$ and this implies, in particular, the existence of two solutions for any $\lambda >0$ sufficiently small.
\end{abstract}
\maketitle



\section{Introduction}

For a bounded domain  $\Omega \subset \R^N, N \geq 3$, with smooth boundary (in the sense of condition (A) of \cite[p.6]{LU68}), we study, depending on the parameter $\lambda \in \R$, the existence and multiplicity of solutions of the  boundary value problem
\begin{equation*}
- \Delta u = \lambda c(x)u+ \mu(x) |\nabla u|^2 + h(x), \quad u \in H^1_0(\Omega) \cap L^{\infty}(\Omega).
\eqno{(P_{\lambda})}
\end{equation*}
Here,  
the  hypotheses are
$$
\hspace{1cm}
\left\{ \begin{array}{c} c \mbox{ and } h \mbox{ belong to }  L^p(\Omega) \quad \mbox{for some } p > \frac{N}{2},
\\[2mm]
c\gneqq 0  \mbox{ and } \mu \in L^{\infty}(\Omega).  
\end{array}
\right.
\leqno{\mathbf{(A1)}}
$$
Observe that problem $(P_\lambda)$ is quasilinear due to the presence 
of the quadratic term  $|\nabla u|^2$. Elliptic quasilinear equations with a gradient dependence up to the critical growth $|\nabla u|^2$
were first studied  by Boccardo, Murat and Puel in the 80's and it has been an active field of research until now. To situate our problem with respect to the existing literature we underline that our solutions are  functions $u\in H_0^1(\Omega)\cap L^{\infty}(\Omega)$ satisfying 
$$
\int_\Omega \nabla u \nabla v \, dx
				= 
\lambda \int_\Omega c(x)  u  v \, dx
		+
\int_\Omega \mu(x)  |\nabla u |^2 v \, dx
		+
\int_\Omega h v \, dx \, , \ \forall v\in H_0^1(\Omega)\cap L^{\infty}(\Omega)\,,
$$
%
and that our problem does not satisfy  the so called {\it sign condition} and thus we cannot follow the approach of \cite{BeBoMu,BoGaMu,BoMuPu1.5}. 

Under the additional condition that  
$c(x) \geq \alpha_0$ a.e. in $\Omega$ for some $\alpha_0 >0$, 
the existence of a solution of $(P_{\lambda})$ when $\lambda <0$  is a special case of the results of \cite{BoMuPu1,BoMuPu2.5,BoMuPu3}. 

Also in the case  $\lambda = 0$ (or equivalently when $c\equiv 0$), 
Ferone and Murat \cite{FeMu1, FeMu2} obtained the existence of a solution for 
$(P_{0})$, under the smallness assumption 
\begin{equation}
\label{FMurat}
 ||\mu||_{\infty} ||h||_{\frac{N}{2}} < \mathcal{S}_N^2,
 \end{equation}
where $\mathcal{S}_N >0$ is the best constant in Sobolev's inequality, namely,
$$  
 \mathcal{S}_N = \inf \left\{ \frac{\left( \int_\Omega  |\nabla \phi|^2\, dx\right)^{1/2} }{ \left(\int_\Omega |\phi |^{2^*}\,dx \right)^{1/{2^*}}}\ : \ \phi\in H_0^1(\Omega)\setminus \{0\}\right\} , \mbox{ with } 2^* = 2N/(N-2).
$$
This result was the first one assuming that $h(x) \in L^{N/2}(\Omega)$ but previous results, in the case $\lambda =0$, were obtained under stronger regularity assumptions on $h(x)$ and assuming that a suitable norm of $h(x)$ is small (see \cite{AlPi, AlLiTr, FePo, FePoRa, GrMo, MaPaSa}).
In the  particular case $\mu(x)\equiv \mu>0$ and $h(x)\geq 0$, this existence result of \cite{FeMu1, FeMu2}  can be  improved  using Theorem 2.3 of  Abdellaoui, Dall'Aglio and Peral in \cite{AbDaPe} (see also \cite{AbBi2}) who show that a sufficient condition for the existence of a solution for $(P_0)$ is 
$$
\mu <
 \displaystyle \inf \left\{ \frac{\int_\Omega |\nabla \phi|^2\,dx}{\int_\Omega h(x)\phi^2\,dx} \, :\, \phi \in H_0^1(\Omega) , \ \int_\Omega h(x)\phi^2\,dx >0\right\}.
$$
 In addition, we remark the interesting result by Porretta \cite{Po} for the case $c(x)\equiv 1$, $\mu(x)\equiv 1$ and $h\in L^\infty (\Omega)$. He   has proved that when the problem $(P_0)$ has no solution, then the solutions of $(P_\lambda)$ for $\lambda <0$ blows-up completely, this behaviour being described in terms of the so-called ergodic problem. \medbreak

Concerning the uniqueness a general theory for problems having quadratic growth in the gradient was developed in \cite{BaBlGeKo,BaMu} (see also \cite{ArSe,BaPo}). When 
$c(x) \geq \alpha_0$ a.e. in $\Omega$ for some $\alpha_0 >0$, the results of \cite{BaBlGeKo} imply the uniqueness of the solutions of $(P_{\lambda})$ when $\lambda <0$. 
For $\lambda =0$, the fact that $(P_{0})$ has at most one solution can also 
 be obtained from  \cite{BaBlGeKo}  provided that either $h(x)$ has a sign or it is sufficient small. See Remark \ref{newrequire1} for more details.
\medbreak

The aim in our first result is twofold. First, we handle functions $c(x)$ that can
vanish in some part of $\Omega$. This does not seem to have been considered in the literature.  Specifically,  for the nonnegative and 
nonzero function $c(x)$ we set 
$$
W_c = \{ w \in H^1_0(\Omega) : c(x) w(x) = 0,  \mbox{ a.e. }  x \in \Omega \},
$$
and, if 
  $\mbox{meas}(\Omega \backslash \mbox{Supp} \, c) > 0$, we assume that
 the following condition holds
$$
\left\{ \begin{array}{c} 
\displaystyle \inf_{ \{ u \in W_c,\, ||u||_{H^1_0(\Omega)}=1 \}} \,  \displaystyle \int_{\Omega} \left(|\nabla u|^2 -  ||\mu^+||_{\infty} h^+(x) u^2 \right) dx >0,
\\
\displaystyle \inf_{ \{ u \in W_c,\, ||u||_{H^1_0(\Omega)}=1 \}} \,  \displaystyle \int_{\Omega} \left(|\nabla u|^2 -  ||\mu^-||_{\infty} h^-(x) u^2 \right) dx >0. 
\end{array}
\right. 
\leqno{\mathbf{(Hc)}}
$$ 
Here $\mu^+ = \max(\mu,0)$, $\mu^- = \max(-\mu,0)$, $h^+ = \max(h,0)$ and $h^- = \max(-h,0)$.    As we shall see  condition  {\rm (Hc)}, along with {\rm (A1)}, suffices to guarantee the existence of a solution of $(P_\lambda)$ for $\lambda<0$.  
Moreover, we   prove that, under {\rm (A1)}, the problem $(P_\lambda)$ for $\lambda \leq 0$ has at most one solution. To obtain this uniqueness result it does not seems possible to extend the approach of \cite{BaBlGeKo,BaMu} and we follow a different strategy. As a first step we establish a regularity result inspired by
\cite{BoMuPu2,Gi, GiMo}
 for the solutions of $(P_\lambda)$. Then, using this regularity we derive our uniqueness result. This approach is applied  directly to problem $(P_\lambda)$. However we believe it can also be used to obtain, under slighty stronger regularity assumptions on the data, new uniqueness results for the general class of problems considered in \cite{BaBlGeKo,BaMu}.
\medskip

Our aim is also to point out that the unique  solution of $(P_\lambda)$ for $\lambda<0$ belongs to  a continuum $C$  whose behavior at $\lambda=0$ depends in an essential way on the existence of solution of $(P_0)$. Throughout the paper we assume that the boundary of $\Omega$ is smooth in the sense of condition (A) of \cite[p.6]{LU68}. Under this assumption it is known,  
\cite[Theorem IX.2.2]{LU68}
that 
any solution of $(P_\lambda)$ belong to $C^{0, \alpha}(\overline{\Omega})$ for some $\alpha >0$.  Denoting the solutions set
$$
\Sigma = \{ (\lambda, u) \in \R \times C(\overline \Omega) 
: (\lambda, u) \mbox{ solves } (P_{\lambda})\},
$$
we prove the following result.
\smallskip

\begin{thm}
\label{negativevalue} 
Assume that  {\rm (A1)} holds.
If in addition, in the case that \\
 $\mbox{\rm meas}(\Omega \backslash \mbox{\rm Supp} \, c) > 0$, 
 we  also assume that {\rm (Hc)} holds, then 
 \medskip
\begin{enumerate}
\item[1)] For $\lambda < 0$, $(P_{\lambda})$ has a unique solution $u_{\lambda}$. \medskip
%
%
%
\item[2)]
  There exists an unbounded continuum $C$ of solutions in 
$\Sigma$ whose projection $\mbox{\rm Proj}_{\R}C$  on the 
$\lambda$-axis contains the interval  $]-\infty,0[$.  \medskip

%
%
\item [3)] Moreover, $\limsup_{\lambda\to 0^-}\|u_\lambda\|_\infty<\infty$ if and only if $(P_0)$ has a  solution. In case $(P_0)$ has a solution $u_0${\color{blue},} it is unique and 
$$
\lim_{\lambda\to 0^-} \|u_\lambda - u_0\|_\infty = 0.
$$
If $(P_0)$ has no solution{\color{blue},} then $\lim_{\lambda\to 0^-}\|u_\lambda\|_\infty = \infty$ and $\lambda=0$ is a bifurcation point from infinity for 
$(P_{\lambda})$ (see Figure~\ref{fig1}).
\end{enumerate}
\end{thm}

\begin{figure}
            \label{fig1}
\scalebox{1} 
{
\begin{pspicture}(0,-2.22)(9.061894,2.22)
\definecolor{color29}{rgb}{0.2,0.2,0.2}
\psline[linewidth=0.04cm,arrowsize=0.05291667cm 2.0,arrowlength=1.4,arrowinset=0.4]{->}(0.16,-1.94)(7.24,-2.0)
\psline[linewidth=0.04cm,arrowsize=0.05291667cm 2.0,arrowlength=1.4,arrowinset=0.4]{->}(5.42,-2.2)(5.44,2.24)
\usefont{T1}{ptm}{m}{n}
\rput(7.461456,-2){$\lambda$}
\usefont{T1}{ptm}{m}{n}
\rput(6.0,2.005){$\|u\|_\infty$}
\psbezier[linewidth=0.04,linecolor=color29](0.0,-1.34)(0.32,-1.32)(0.92,-1.42)(1.96,-1.18)(3.0,-0.94)(3.6,-0.94)(4.34,-0.02)(5.08,0.9)(5.12,0.92)(5.3,2.2)
\usefont{T1}{ptm}{m}{n}
\rput(4.5614552,0.905){$C$}
\end{pspicture} 
}
\caption{Bifurcation diagram when $(P_0)$ has no  solution}
\end{figure}
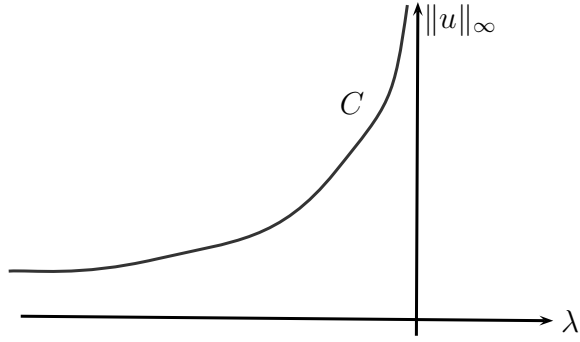

\begin{remark}\label{comparaison}
Condition {\rm (Hc)} connects the two limit cases: 
$c(x) \geq \alpha_0 >0$ and $c\equiv 0$ ($\lambda=0$). If $c(x) >0$ a.e. on $\Omega$ we have $\mbox{meas}(\Omega \backslash \mbox{Supp} \, c) =0$. 
Thus, 
under {\rm (A1)}, a solution of $(P_{\lambda})$ exists for any $\lambda <0$. If $\mbox{meas}(\Omega \backslash \mbox{Supp} \, c)  >0$, the situation is more delicate.
When both $\mu(x) \geq 0$ and $h(x) \geq 0$, {\rm (Hc)} relates the {\it size} of $\mu(x)h(x)$ to the {\it size} of $\Omega \backslash \mbox{Supp} \, c$, 
showing 
that the signs of  $\mu(x)$ and $h(x)$ with respect to one another strongly influence the existence of solution of $(P_{\lambda})$ when $\lambda <0$. Indeed, {\rm (Hc)} holds if either $\mu(x) \geq 0$ and $h(x) \leq 0$ a.e. in $\Omega$, or $\mu(x) \leq 0$ and $h(x) \geq 0$ a.e. in $\Omega$.  Moreover, 
it holds true under condition \eqref{FMurat} since, from the Sobolev embedding, it follows that
$$
\int_{\Omega} h(x) v^2 dx \leq ||h||_{N/2}||v||^2_{2^*} \leq \frac{1}{\mathcal{S}_N^2}||h||_{N/2}||\nabla v||_2^2.
$$
Hence we obtain the above refered results as  a corollary.  In  Remark \ref{ktL2} we show that {\rm (Hc)} is somehow sharp for the existence of solution of $(P_{\lambda})$.
\end{remark}

\medbreak

\begin{remark}
We shall also prove, in Corollary \ref{prop3a}, that a sufficient condition for the existence of solution of $(P_0)$ is that 
condition {\rm (Hc)} is satisfied with $c(x)\equiv 0$, i.e that the  following condition is hold
$$
\left\{ \begin{array}{c} 
\displaystyle \inf_{ \{  u \in H^1_0(\Omega),\, ||u||_{H^1_0(\Omega)}=1 \}} \,  \displaystyle \int_{\Omega} \left(|\nabla u|^2 -  ||\mu^+||_{\infty} h^+(x) u^2 \right) dx >0,
\\
\displaystyle \inf_{ \{  u \in H^1_0(\Omega),\, ||u||_{H^1_0(\Omega)}=1 \}} \,  \displaystyle \int_{\Omega} \left(|\nabla u|^2 -  ||\mu^-||_{\infty} h^-(x) u^2 \right) dx >0. 
\end{array}
\right.
\leqno{\mathbf{(H0)}} 
$$ 
%
\end{remark}

Our next result show that the existence of a solution of $(P_0)$  suffices to guarantee the existence of a continuum of solutions $C \subset \Sigma$ such that $\mbox{Proj}_{\R}C$ contains $]-\infty,a]$ for some $a>0$.

\begin{thm}\label{th1.1a}
Assume {\rm (A1)} and suppose that $(P_0)$ has a solution.  Then
\begin{enumerate}
\item[1)] For all $\lambda \leq 0$, $(P_\lambda)$ has a, unique, solution $u_\lambda$. \medskip

\item[2)] There exists a continuum $C\subset\Sigma$ such that \smallskip
\begin{enumerate}
	\item $\{(\lambda,u_\lambda):\, \lambda\in \,]-\infty,0]\, \}\subset C$. \medskip
	\item $C\cap ([0,\infty[\,\times C(\overline{\Omega}))$ is a unbounded set in
$\R\times C(\overline{\Omega})$.
\end{enumerate}
\smallskip
In particular, $\mbox{Proj}_{\R}C$ contains $]-\infty,a]$ for some $a>0$.
\end{enumerate}
\end{thm}
\medskip

Finally, in the last part of the paper and under stronger assumptions, we study the behaviour in the half space $\{ \lambda >0 \} \times C(\overline \Omega)$ of the branch $C \subset \Sigma$ obtained in Theorem \ref{th1.1a}  and we obtain a multiplicity result.
\vspace{0.3cm}

First we note that, in case $\mu\equiv 0$, we cannot have multiplicity results except when $\lambda$ is an eigenvalue of the problem
\begin{equation}
\label{eigenvaluep}
-\Delta \varphi_{1} = \gamma c (x) \varphi_{1},  \quad \varphi_1 \in H^1_0(\Omega),
\end{equation}
and $h(x)$ satisfies the {\it ``good''} orthogonality condition. Hence, there is no hope to obtain multiplicity results just under our assumption {\rm (A1)}.
\vspace{0.3cm}

Multiplicity results have been considered by Abdellaoui, Dall'Aglio and  Peral \cite{AbDaPe} (see also \cite{AbBi2, Si}) for $(P_{\lambda})$ in the case $\lambda=0$ and when $\mu(x)$ is replaced by some $g(u)$ satisfying $u g(u)<0$. In a recent paper, Jeanjean and Sirakov \cite{JeSi} study the case $\lambda>0$ when
 $\mu(x)$ is a positive constant but $h(x)$ may change sign and satisfy a condition related to   \eqref{FMurat}. Using Theorem 2 of 
\cite{JeSi} an explicit $\lambda_0 >0$ can be derived under which $(P_{\lambda})$ has two solutions whenever $\lambda \in \,]0, \lambda_0[$. 
 
\vspace{0.3cm}
The above quoted multiplicity  results have the common property that the coefficient of $|\nabla u|^2$ (either $g(u)$ or the constant $\mu$) does not depend on $x$. This allows the authors to make a change a variable, similar to the one used in \cite{KaKr}, in order to transform the problem in a semilinear one (i.e. without gradient dependence). Then variational methods are used to prove multiplicity results on the transformed problem. In our case, we consider problem $(P_{\lambda})$ with a non constant function coefficient $\mu(x)$, which implies that this change of variable is no more possible.
\medbreak

We replace {\rm (A1)} by the stronger assumption
$$
\left\{ 
\begin{array}{c} 
c \mbox{ and } h \mbox{ belongs to }  L^p(\Omega) \quad \mbox{for some } p > \frac{N}{2},
\\[2mm]
   c\gneqq 0, \, h \gneqq 0 \mbox{ and } \mu_2\geq \mu(x) \geq \mu_1 \mbox{ for some } \mu_2\geq\mu_1 >0.  
\end{array}
\right.
\leqno{\mathbf{(A2)}}
$$
Let $\gamma_1 >0$  denote the first eigenvalue of the problem  (\ref {eigenvaluep}). We prove the following theorem.

\begin{thm}
\label{th3}
Assume $\mathrm{(A2)}$ and suppose that $(P_0)$ has a solution.
Then the continuum $C \subset \Sigma$ obtained in Theorem \ref{th1.1a} consists of non negative functions{\color{blue},} 
its projection  $ \mbox{\rm Proj}_{\R} C$ 
on the $\lambda$-axis is an unbounded interval  $]-\infty,\overline \lambda] \subset {]-\infty,\gamma_1[}$ containing $\lambda =0$  and  $C \subset \Sigma$ bifurcates from infinity to the right of the axis $\lambda = 0$. 
Moreover, there exists $\lambda_0\in {]0,\overline \lambda]}$
such that
for all $\lambda \in {]0, \lambda_0[}$, the section $C\cap (\{\lambda\}\times C(\overline\Omega))$
 contains two distinct non negative solutions of $(P_{\lambda})$ in $\Sigma$ 
 (see Figure 2).
\end{thm}

\begin{remark}
In order to prove Theorem  \ref {th3} the key points are the observation that the continuum cannot cross the line $\lambda = \gamma_1$ and the derivation of a priori bounds, for any $a>0$, on the (positive) solutions of $(P_{\lambda})$ for 
$\lambda \in\, ]a, \gamma_1]$. These a priori bounds are obtained by an extention of the classical approach of Brezis and Turner \cite{BT77}.
\end{remark}

\begin{remark}
The fact that on a MEMS type equation, involving a critical term in gradient, a multiplicity result had also been derived throught the study of the behaviour of a continuum \cite{WeYe} was pointed out to us by D. Ye after the completion of the present work.
\end{remark}

The paper is organized as follows. In Section \ref{Section1} we recall some results concerning the method of lower and upper solutions as well as a continuation theorem. In Section \ref{Section02} we derive various existence results for problems of the type of $(P_{\lambda})$ when $\lambda \leq 0$. Section \ref{Sectionuniqueness-0} deals with the uniqueness issue. In Section \ref{Sectionuniqueness} we establish the existence of a continuum of solutions. Section \ref{Section2} is devoted to the study of the branch in the half space $\{ \lambda >0\} \times C(\overline{\Omega})$ and in particular to the derivation of a priori bounds, see Proposition
\ref{bounds1}. The proofs of our three theorems are given in Section \ref{Proofs}. Finally a technical result, Lemma \ref{comp}, is proved in Section \ref{appendix}.




\begin{figure} 
\label{fig22}
\scalebox{1} 
{
\begin{pspicture}(0,-2.07)(7.2,2.09)
\psline[linewidth=0.04cm,arrowsize=0.05291667cm 2.0,arrowlength=1.4,arrowinset=0.4]{->}(0.0,-1.51)(5.34,-1.47)
\psline[linewidth=0.04cm,arrowsize=0.05291667cm 2.0,arrowlength=1.4,arrowinset=0.4]{->}(0.86,-2.05)(0.88,2.01)
\psbezier[linewidth=0.04
](0.04,-1.11)(1.14,-1.09)(3.6,-0.65)(2.72,-0.01)(1.84,0.63)(2.64,0.11)(1.88,0.61)(1.12,1.11)(1.1,1.05)(1.02,1.93)
\psline[linewidth=0.04cm,linestyle=dashed,dash=0.16cm 0.16cm](2.98,-0.35)(2.98,-1.47)
\usefont{T1}{ptm}{m}{n}
\rput(2.91,-1.725){$\lambda_0$}
\usefont{T1}{ptm}{m}{n}
\rput(5.6,-1.725){$\lambda$}
\usefont{T1}{ptm}{m}{n}
\rput(0.34,1.895){$\| u\|_\infty$}
\usefont{T1}{ptm}{m}{n}
\rput(2,0.9){$C$}
\end{pspicture}
}
\caption{Bifurcation diagram when $(P_0)$ has a solution.} 
\end{figure}
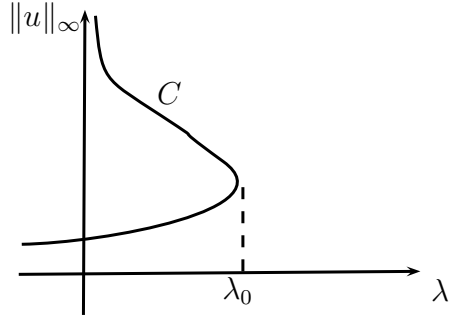

\vskip9pt
\begin{center}
\textbf{Notation.}
\end{center}
\begin{enumerate}{\small
\item $\Omega \subset \R^N$, $N \geq 3$ is a bounded domain whose boundary $\partial \Omega$ is sufficiently regular as to satisfies the condition (A) of \cite[p.6]{LU68}. A sufficient condition for (A) is that $\partial \Omega$ satisfies the exterior uniform cone condition.
\item For any measurable set $\omega \subset \R^N$ we denote by $\mbox{meas}(\omega)$ its Lebesgue measure.
\item  For $p\in [1,+\infty[$, the norm $(\int_{\Omega}|u|^pdx)^{1/p}$
in $L^p(\Omega)$ is denoted by $\|\cdot\|_p$. We denote by $p^{\prime}$
the  conjugate exponent of $p$, namely $p^{\prime} = p/(p-1).$ The norm in $L^{\infty}(\Omega)$ is $\|u\|_{\infty}=\mbox{esssup}_{x\in \Omega}|u(x)|$.
\item For $v \in L^1(\Omega)$ we define $v^+= \max(v,0)$ and $v^- =
\max(-v,0)$.
\item For $h\in L^1(\Omega)$ we denote $h\gneqq 0$ if $h(x)\geq 0$ for a.e. $x\in\Omega$ and $\mbox{meas}(\{x\in\Omega:
h(x)>0\})>0$.
\item We denote by $H$ the space $H^1_0(\Omega)$ equipped with the Poincar\'{e}
norm $||u||:=\left( \int_\Omega |\nabla u|^2\,dx\right)^{1/2}$.
\item We denote by $X$ the space $H_0^1(\Omega) \cap L^{\infty}(\Omega)$.
\item We denote by $B_r(u_0) $ the ball $ \{u \in H\, :\, ||u-u_0||<r\}.$
\item We denote by $C,D>0$ any positive constants which are not
essential in the problem and may vary from one line to another.
}
\end{enumerate}


\section{Preliminaries}
\label{Section1}

In our proofs we shall use the method of lower and upper solutions. We present here Theorem 3.1 of \cite{BoMuPu4} adapted to our setting. We consider the
boundary value problem
\begin{equation}
\label{bo-Mu-Pu}
- \Delta u +H(x,u, \nabla u)=f, \quad u \in X
\end{equation}
where $f\in L^1(\Omega)$ and  $H$ is a Carath\' eodory function 
from $\Omega \times \R \times \R^N$ into $\R$ with a natural growth, i.e., for which there exist a nondecreasing function $b$ from $[0,+\infty[$ into $[0,+\infty[$ and $k \in L^1(\Omega)$ such that, for a.e. $x\in \Omega$ and all $(u,\xi)\in \mathbb R\times\mathbb R^{N}$,
$$
|H(x,u, \xi)| \leq b(|u|) [k(x) + |\xi|^2].
$$
We also recall (see \cite{BoMuPu4}) that a {\it lower solution} (respectively, an {\it upper solution}) of (\ref{bo-Mu-Pu}) is a function $\alpha$ (respectively, $\beta$) $\in H^1(\Omega) \cap L^{\infty}(\Omega)$ such that
$$
- \Delta \alpha + H(x,\alpha, \nabla \alpha)\leq f(x) \mbox{ in } \Omega, \quad \alpha \leq 0 \mbox{ on } \partial \Omega,
$$
(respectively,
$$
- \Delta \beta +H(x,\beta, \nabla \beta)\geq f(x) \mbox{ in } \Omega, \quad \beta \geq 0 \mbox{ on } \partial \Omega).
$$ 
This has to be understood 
in the sense that $\alpha^+\in H^1_0(\Omega)$ and
$$
\int_{\Omega} \nabla\alpha\nabla v \,dx+\int_{\Omega} H(x,\alpha,\nabla \alpha) v\,dx\leq \int_{\Omega} f(x)v\,dx,
$$
(respectively, $\beta^-\in H^1_0(\Omega)$ and 
$
\int_{\Omega} \nabla\beta\nabla v \,dx+\int_{\Omega} H(x,\beta,\nabla \beta) v\,dx\geq \int_{\Omega} f(x)v\,dx
$), for all $v\in H^1_0(\Omega)\cap L^\infty (\Omega)$ with $v\geq 0$ a.e. in $\Omega$.

\begin{thm}[Boccardo-Murat-Puel \cite{BoMuPu4}]
\label{sousBo}
If there exist a lower solution $\alpha$ and an upper solution $\beta$ of {\rm(\ref{bo-Mu-Pu})} with $\alpha \leq \beta$ a.e. in $\Omega$, then there exists a solution $u$ of {\rm(\ref{bo-Mu-Pu})} with $\alpha \leq u \leq \beta$ a.e. in $\Omega$.
\end{thm}
\smallskip

We  also need a continuation theorem. Let $E$ be a real Banach space with norm $||\cdot||_E$  and $T:\R \times E \to E$ a  completely continuous map, i.e. it is continuous and maps bounded sets to relatively compact sets. For $\lambda \in \R$, we consider the problem of finding the zeroes of $\Phi(\lambda, u) := u - T(\lambda, u)$, i.e.
\begin{equation*}
\Phi(\lambda, u) = u - T(\lambda, u) = 0, \quad u \in E \eqno{(Q_{\lambda})},
\end{equation*}
and we define
$$
\Sigma = \{ (\lambda, u) \in \R \times E : \Phi(\lambda, u) = 0 \}.
$$ 
Let $\lambda_0 \in \R$ be arbitrary but fixed and for $v \in E$ and $r>0$ let $B(v,r):= \{u \in E : ||v- u|| <r \}$. 

If we assume that $u_{\lambda_0}$ is an isolated solution of  $(Q_{\lambda_0})$,  then the Leray-Schauder degree $\mbox{deg}(\Phi(\lambda_0,\cdot), B(u_{\lambda_0},r), 0)$ is well defined  and is constant for $r>0$ small enough. Thus it is possible to define the index
$$
i(\Phi(\lambda_0, \cdot), u_{\lambda_0}) := \lim_{r \to 0} \mbox{deg}(\Phi(\lambda_0, \cdot), B(u_{\lambda_0},r), 0).
$$


\begin{thm} 
\label{1noacot}
If $(Q_{\lambda_0})$ has a unique solution $u_{\lambda_0}$ and $i(\Phi(\lambda_0, \cdot), u_{\lambda_0}) \neq 0$ then $\Sigma$ possesses two unbounded components $C^+$, $C^-$ in $[\lambda_0, + \infty[ \times E$ and $]- \infty, \lambda_0] \times E$ respectively which meet at $(\lambda_0, u_{\lambda_0})$.
\end{thm}

Theorem \ref{1noacot}
is essentially Theorem 3.2 of \cite{Ra} (stated assuming that $\lambda_0=0$). In turn this result is essentially due to Leray and Schauder \cite{LeSc}. For an exposition of the main properties of the Leray-Schauder degree see, for example, \cite{AmAr}.
\vspace{3mm}

\section{Some existence results} 
\label{Section02}

In this section we establish some existence results for the boundary value problem
\begin{equation}
\label{31}
- \Delta u = d(x)u+ \mu(x) |\nabla u|^2 + h(x), \quad u \in X
\end{equation}
under the assumption that
$$
\leqno{\mathbf{(A3)}} \hspace{1cm}
\left\{ \begin{array}{c} d \mbox{ and } h \mbox{ belong to }  L^p(\Omega) \quad \mbox{for some } p > \frac{N}{2},
\\[2mm]
   \mu(x)\equiv \mu>0 \mbox{ is a constant},
\\[2mm]
   d \leq 0\mbox{ and } h \geq 0.  \end{array}
\right.
$$
If $\mbox{meas}(\Omega \backslash \mbox{Supp} \, d) > 0$ we also set
$$
W_d = \{ w \in H^1_0(\Omega) : d(x) w(x) = 0, \, a.e. \, x \in \Omega \}
$$
and we impose condition $({\rm Hc})$ for $c=d$, i.e., we require 
$$
\leqno{\mathbf{(H)}} \hspace{1cm}
  m_2 := \displaystyle \inf_{ \{ u \in W_d,\, ||u||=1  \}} \, \int_{\Omega} (|\nabla u|^2 - \mu h(x) u^2)\, dx >0.
$$

\begin{prop}
\label{prop1}
Assume  {\rm (A3)} and, if $\mbox{\rm meas}(\Omega \backslash \mbox{\rm Supp} \, d) > 0$,  also that {\rm (H)} holds.
Then {\rm(\ref{31})} has a non negative solution.
\end{prop}

\begin{remark}\label{pos}
Observe that, under condition {\rm (A3)}, every solution $u$ of {\rm(\ref{31})} is non negative. In fact, using $u^-$ as test function we obtain, as $d\leq 0$, $\mu> 0$ and $h\geq 0$,
$$
0\geq  -\int_\Omega |\nabla u^-|^2 + \int_\Omega d(x) |u^-|^2 = \int_\Omega \left[ \mu |\nabla u|^2 + h(x)\right] u^{-} \geq 0,
$$
which implies that $u^-=0$ i.e. $u\geq 0$.
\end{remark}

\begin{remark}
\label{ktL1} 
Assume, in addition to {\rm (A3)}, that   there exists an open subset $O(d)$ in $\Omega$ with $C^1$ boundary $\partial O(d)$  such that 
	$d(x)=0$ a.e. in $\overline{O(d)}$ and $d(x)<0$ a.e. in $ \Omega\setminus O(d)$.
  Then  \eqref{31} has a solution if and only if {\rm (H)} holds.   See Remark \ref{ktL2} below for 
a proof.
\end{remark}

To prove Proposition \ref{prop1} we introduce the boundary value problem 
\begin{equation}
\label{3}
- \Delta v - \mu h(x) v = d(x) g(v) + h(x),\quad  v \in H 
\end{equation}
where
\begin{equation}
\label{defg}
\begin{array}{ll}
g(s) = 
\left\{
\begin{array}{ll} 
\text{$\frac{1}{\mu} (1+ \mu s) \ln (1+ \mu s)$},   &\mbox{if } \quad s \geq 0,
\\
\text{$-\frac{1}{\mu} (1-\mu s) \ln (1- \mu s)$},   &\mbox{if } \quad s < 0.
\end{array}
\right.
\end{array}
\end{equation}
Let us denote 
$$
G(s) =\int_0^s g(\xi)\,d\xi = \left\{ \begin{array}{ll}
               \displaystyle \frac{{(1+\mu s)^2}}{4 \mu^{2}}  [2\ln (1+\mu s) -1] {+ \frac{1}{4 \mu^2}} & \mbox{ if } s\geq 0,
               \\
               \\
               G(-s), &  \mbox{ if } s< 0.
               \end{array}
               \right.
               $$

The properties of $g$ that are useful to us are gathered in the following lemma.

\begin{lem}  
\label{prop-g1}
\vskip2pt \noindent
\begin{itemize}
\item[(i)] The function $g$ is odd and continuous on $\R$.
\item[(ii)] $g(s)s >0$ for  $s \in \R \setminus \{0\}$, 
$G(s) \geq 0$ on $\R$.
\item[(iii)] For any $r\in\,]0,1[$, there exists $C=C(r,\mu) >0$  such that, for all $|s| > \frac{1}{\mu}$,  we have $|g(s)| \leq C |s|^{1+r}$.
\item[(iv)] We have  $G(s)/s^2 \to + \infty$ as $|s| \to  \infty$. \qed
\end{itemize}
\end{lem}


The idea of modifying the problem to obtain  problem (\ref{3}) is not new. It appears already in \cite{KaKr}  in another context. It permits to obtain a non negative solution  of~\eqref{31}.

\begin{lem}
\label{dual1}
Assume that {\rm (A3)} hold.
\begin{enumerate}
\item[i)] Any solution of \eqref{3} belongs to $W^{2,p}(\Omega)$ and thus to $L^{\infty}(\Omega)$;
\item[ii)] If $v \in H$ is a non negative solution of \eqref{3} then $u = (1/ \mu) \ln (1 + \mu v)\in X$ is a (non negative) solution of
\eqref{31}.
\end{enumerate}
\end{lem}

\begin{proof}
%
i) Let $v \in H$ be a solution of~\eqref{3}, that we write as
$$ 
- \Delta v = \left[\mu h(x) + d(x) \frac{g(v)}{v} \right] v +  h(x), \quad v\in H.
$$
By classical arguments, see for example  \cite[Theorem III-14.1]{LU68}, as $\partial \Omega$ satisfies the condition (A) of \cite{LU68},
 the first part of the lemma will be proved if we can show that
$$
\left[\mu h(x) + d(x) \frac{g(v)}{v} \right] \in L^{p_1}(\Omega) \quad \mbox{with } p_1>N/2.
$$
But by assumption $d$ and $\mu h$ belong to
$L^{p}(\Omega)$, for some $p>N/2$ and we shall prove that the term $d(x) \frac{g(v)}{v}$ has the same property.
This is the case
because of the slow growth of $g(s)/s$ as $|s| \to \infty$, see Lemma \ref{prop-g1}-(iii). Specifically, for any $r\in \,]0,1[$, there
exists a $C >0$ such that, for all $|s|> \frac{1}{\mu}$,
$$
\left|g(s)/s \right| \leq C |s|^r.
$$
Thus, since $d \in L^p(\Omega)$ with $p > \frac{N}{2}$ and  $v\in L^{\frac{2N}{N-2}}(\Omega)$,  taking
$r >0$ sufficiently small (for example $r <
\frac{4p-2N}{p(N-2)}$) we see, using  H\"older inequality, that  $d(x) g(v)/v \in L^{p_1}(\Omega)$, for some $p_1 > N/2$. This ends the proof of Point i). 
\medskip

ii) Since $v\geq 0$ the problem \eqref{3} can be rewritten as
\begin{equation}
\label{2.12}
- \Delta v = \frac{d(x)}{\mu}(1 + \mu v) \ln (1 + \mu v) + (1 + \mu v) h(x), \qquad v\in H.
\end{equation}
Let $v \in H$ be a non negative solution of \eqref{2.12}, we want to show that
$ u = \frac{1}{\mu}\ln(1+ \mu v)$ is a  solution
of~\eqref{31}, namely that, for $\phi \in
C_0^{\infty}(\Omega)$, 
\begin{equation}
\label{2.13}
\int_{\Omega} \left(\nabla u \nabla \phi - \mu |\nabla u|^2 \phi -
d(x) u \phi\right)  dx = \int_{\Omega} h(x) \phi \thinspace
dx.
\end{equation}
First observe that, as $v\in L^{\infty}(\Omega)$ and satisfies $v\geq 0$  in $\Omega$ we have $u\in X$.
Let $\displaystyle \psi = \frac{\phi}{1+ \mu v}.$ Clearly $\psi
\in H$ and thus it can be used as test function in 
\eqref{2.12}. Hence, we get
\begin{equation}
\label{2.15}
\begin{array}{rcl}
\displaystyle
\int_{\Omega} \nabla v \nabla \psi \thinspace dx 
&=&
\displaystyle 
\int_{\Omega}
\frac{d(x)}{\mu} \ln(1+ \mu v) \phi \thinspace dx +
\int_{\Omega}h(x) \phi \thinspace dx
\\[3mm]
&=&
\displaystyle
\int_{\Omega}d(x) u \phi \thinspace dx
+
\int_{\Omega}h(x) \phi \thinspace dx.
\end{array}
\end{equation}
Moreover, we have
\begin{eqnarray}
\nonumber 
\int_{\Omega} \nabla v \nabla \psi \thinspace dx  
& = &
\int_{\Omega}\nabla \left( \frac{1}{\mu}(e^{\mu u}-1)\right)\nabla \left(\frac{\phi}{1+ \mu v}\right)dx 
\\
\nonumber
& = &     
\int_{\Omega} e^{\mu u} \nabla u \left( \frac{\nabla \phi}{1+ \mu v}-
\frac{\mu \phi \nabla v }{(1+ \mu v)^2}\right) dx 
\\
\nonumber 
& = & \int_{\Omega} \nabla u \left( \nabla \phi -
\frac{\mu \phi \nabla
(\frac{1}{\mu}(e^{\mu u}-1))}{(1+ \mu v)} \right)dx 
\\
& =&  \int_{\Omega }\nabla u (\nabla \phi - \mu \phi
\nabla u ) \thinspace dx \nonumber 
\\ 
& =&  \int_{\Omega} \left(\nabla u \nabla \phi -
\mu |\nabla u|^2 \phi\right) dx. \nonumber
\end{eqnarray}
Combining this equality with \eqref{2.15}  we see that $u$
satisfies \eqref{2.13}. This ends the proof of Point ii).
\end{proof}

In order to find a solution of (\ref{3}) we shall look to a critical point of the functional $I$ defined on $H$ by
$$
I(v) = \frac{1}{2} \int_{\Omega} (|\nabla v |^2 - \mu h(x) v^2)  \thinspace dx 
- \int_{\Omega} d(x) G(v)  \thinspace dx 
- \int_{\Omega}h(x) v \thinspace dx.
$$

As $g$ has a subcritical growth at infinity, see Lemma \ref{prop-g1}(iii), it is standard to show that $I \in C^1(H, \R)$ and that a critical point of $I$ corresponds to a solution in $H$ of (\ref{3}).  To obtain a critical point of $I$ we shall prove the existence of a global minimum of $I$. We define
\begin{equation}
\label{min}
m:= \inf_{u \in H} I(u) \in \R \cup \{- \infty\}.
\end{equation}

\begin{lem}
\label{minbounded}
Assume  {\rm (A3)} and, if $\mbox{\rm meas}(\Omega \backslash \mbox{\rm Supp} \, d) > 0$, assume also that {\rm (H)} holds. Then the infimum $m$ defined by  \eqref{min} is finite and it
 is reached by a non negative function  in $H$. Consequently, \eqref{3} has a non negative solution.
\end{lem}

\begin{proof} 
We divide the proof into two steps : \medskip

\noindent {\bf Step 1.} {\it $I$ is coercive. } \medskip

We assume by contradiction the existence of a sequence $\{v_n\} \subset H$ such that $\|v_n\|\to\infty$ and $I(v_n)$ is bounded from above. We define
$$
w_n = \frac{v_n}{||v_n||}.
$$
Clearly $||w_n||\equiv 1$ and we can assume that
$w_n \rightharpoonup w $ weakly in $ H $ and $ w_n \to w $ strongly in $ L^q(\Omega) $ for $ q \in [2, \frac{2N}{N-2}[.$ 
Since $I(v_n)$ is bounded from above,  we have
\begin{equation}
\label{6}
\limsup_{n \to \infty} \frac{I(v_n)}{||v_n||^2} \leq 0.
\end{equation}

We shall treat separately the two cases : 
$$ 
{(1)} \,\,\, w \in W_d \quad \mbox{ and } \quad  {(2)} \,\,\, w \not \in W_d.
$$

\noindent{\it Case $(1)$: $w \in W_d$.}
In this case, taking {\rm (H)} into account, it follows that 
\begin{equation*}
\label{10}
\int_{\Omega}(|\nabla w|^2 - \mu h(x) w^2) dx  \geq m_2||w||^2.
\end{equation*}
Thus,   and since $G(s) \geq 0$ on $\R$ and $d(x)\leq 0$ in $\Omega$, using the weak lower semicontinuity of $\int_{\Omega}|\nabla u|^2dx$ and the weak convergence of $w_n$, we obtain
\arraycolsep1.5pt
\begin{eqnarray}\label{1000}
\liminf_{n \to \infty} \frac{I(v_n)}{||v_n||^2} 
&=&  \liminf_{n \to \infty} \left[\frac{1}{2} \int_{\Omega} (|\nabla w_n|^2  - \mu h(x) w_n^2) dx - \int_{\Omega} \frac{d(x) G(v_n)}{||v_n||^2} dx \right] \nonumber
\\
&\geq & \frac{1}{2} \int_{\Omega}(|\nabla w|^2 - \mu h(x)w^2) dx 
\geq  \frac{1}{2} m_2 ||w||^2 \geq 0 \geq \limsup_{n \to \infty} \frac{I(v_n)}{||v_n||^2},
\end{eqnarray}
\arraycolsep5pt 
i.e., $\lim_{n \to \infty} \frac{I(v_n)}{||v_n||^2} =0$ and  $w \equiv 0$. However,  using that $2p/(p-1) < 2N/(N-2)$ and $w_n$ is weakly convergent to $w=0$ in $H$, we deduce  the strong convergence of $w_n$ to $w=0$ in $L^{2p/(p-1)}(\Omega)$, which by  the assumptions  $d(x) \leq 0$ on $\Omega$ and $G(s) \geq 0$ on $\R$ implies that
$$
\lim_{n \to \infty}  \frac{I(v_n)}{||v_n||^2} \geq \frac{1}{2} - \lim_{n \to \infty}  \frac{\mu}{2} \int_{\Omega}  h(x) w_n^2 dx - \lim_{n \to \infty}  \int_{\Omega} \frac{h(x)w_n}{||v_n||} dx \geq  \frac{1}{2} .
$$
This is a contradiction showing that case {(1)} cannot occurs.  
\medbreak

\noindent{\it Case $(2)$: $w \not \in W_d$.}
Since $w \not \in W_d$, necessarily $\Omega_0=\{x\in \Omega,\, d(x) w(x)\not=0\}$ has non zero measure and thus $|v_n(x)|=|w_n(x)|\,\|v_n\|\to\infty$ a.e. in $\Omega_0$.  Using the assumptions $d(x) \leq 0$ in $\Omega$ and $G(s) \geq 0$ on $\R$ we deduce from Lemma \ref{prop-g1}-(iv) and Fatou's lemma that
\begin{eqnarray*}
\limsup_{n\to\infty}\int_{\Omega} \frac{d(x)G(v_n)}{v_n^2}w_n^2 dx  
&\leq &
 \limsup_{n\to\infty} \int_{\Omega_0}  \frac{d(x)G(v_n)}{v_n^2}w_n^2 dx
 \\ 
 &\leq& 
\int_{\Omega_0}  \limsup_{n\to\infty} \frac{d(x)G(v_n)}{v_n^2}w_n^2 dx  =-\infty.
\end{eqnarray*}

 On the other hand, using that $w_n$ is weakly convergent in $H$ and that,  by Sobolev's embedding, $||w_n||_{\frac{2p}{p-1}}$ is bounded, it follows 
that
$$
0\geq \limsup_{n \to \infty} \frac{I(v_n)}{||v_n||^2} \geq \liminf_{n \to \infty} \frac{I(v_n)}{||v_n||^2}  \geq -C 
- \limsup_{n \to \infty} \int_{\Omega} \frac{d(x) G(v_n)}{||v_n||^2} dx 
=+\infty ,
$$
a contradiction proving that case {(2)} does not occur and the proof of Step 1 is concluded.

\medskip

\noindent{\bf Step 2.} \, {\it Existence of a minimum of $I$.} \medskip

To show that $I$ admits a global minimizer it now suffices to show that $I$ is weakly lower semicontinuous i.e., if  $\{v_n\} \subset H$ is a sequence 
such that $v_n \rightharpoonup v$ weakly in $H$, and then  $v_n \to v$ strongly in $L^q(\Omega)$ for $q \in [2,\frac{2N}{N-2}[$, we have
\begin{equation}
\label{155}
I(v) \leq \liminf_{n \to \infty}I(v_n).
\end{equation}
Using the weak convergence of the sequence $\{v_n\}$ and the weak lower semicontinuity of $\int_{\Omega}|\nabla v|^2 dx$, we have
\begin{equation}
\label{166}
\frac{1}{2}\int_{\Omega}|\nabla v|^2 dx  - \int_{\Omega}h(x)v dx  \leq \liminf_{n \to \infty} \left[\frac{1}{2}\int_{\Omega} |\nabla v_n|^2 dx - \int_{\Omega} h(x) v_n dx\right].
\end{equation}
Also,  the strong convergence in $L^{\frac{2p}{p-1}}(\Omega)$ implies that
\begin{equation}
\label{177}
\int_{\Omega}\mu h(x) v_n^2 dx \to \int_{\Omega} \mu h(x) v^2 dx.
\end{equation}
Finally, since $-d(x) G(v_n) \geq 0$ on $\Omega$, as a consequence of Fatou's lemma, we obtain
\begin{equation}
\label{188}
\int_{\Omega} - d(x) G(v) dx \leq \liminf_{n \to \infty} \int_{\Omega} -d (x) G(v_n) dx.
\end{equation}
At this point (\ref{155}) follows from (\ref{166})-(\ref{188}).
\medbreak

\noindent{\bf Step 3.} \, {\it Conclusion.} \medskip

 
To conclude the existence of a non negative minimum, observe that, as $h(x)\geq 0$ in $\Omega$ and $G(s)$ is even we have, for every $u\in H$,
$$
I(|u|)\leq I(u),
$$
and hence if $v\in H$ is a minimum of $I$ then $|v|$ is also a minimum. Then we conclude that the infimum $m$ is reached by a non negative function.
\end{proof}

\begin{proof}[Proof of Proposition \ref{prop1}]
By Lemma \ref{minbounded}, (\ref{3}) admits a non negative solution  $v \in H$ and thus, using Lemma  \ref{dual1}, we deduce that (\ref{31}) has a non negative solution.
\end{proof}

We now consider the problem.
\begin{equation}
\label{eq22}
- \Delta u =  d(x)u+ W(x,u, \nabla u), \quad u \in X,
\end{equation}
where we assume
\vspace{3mm}
$$
\leqno{\mathbf{(A4)}} 
\left\{ \begin{array}{c} d\leq 0  \mbox{ with 
}  d \in L^p(\Omega)  \mbox{ for some } p > \frac{N}{2} 
\\[2mm]
     \mbox{  and there exist } \mu_{\pm} \in \,]0,+\infty[  \mbox{ and } h_{\pm} \in L^p(\Omega) \mbox{ with } h_{\pm} \geq 0, \mbox{ such that}
    \\[2mm]  
    -\mu_-|\xi |^2-h_-(x)\leq W(x,u, \xi)  \leq \mu_+ |\xi|^2 + h_+(x) \mbox{ on }\Omega\times\R\times\R^N. 
    \end{array}
\right.
$$ 

\begin{prop}
\label{prop2}
Assume that {\rm (A4)}  holds and, if $\mbox{\rm meas}(\Omega \backslash \mbox{\rm Supp} \, d) >0$, in addition, assume
$$
\left\{ \begin{array}{c} 
\displaystyle \inf_{ \{ u \in W_d,\, ||u||=1 \}} \,  \displaystyle \int_{\Omega} \left(|\nabla u|^2 -  \mu_+ h_+(x) u^2 \right) dx >0,
\\
\displaystyle \inf_{ \{ u \in W_d,\, ||u||=1 \}} \,  \displaystyle \int_{\Omega} \left(|\nabla u|^2 -  \mu_-  h_-(x) u^2 \right) dx >0. 
\end{array}
\right.
$$ 
Then \eqref{eq22} has a solution.
\end{prop}

\begin{proof}
To prove Proposition \ref{prop2} we use Theorem \ref{sousBo}. Thus we need to find a couple of lower and upper solutions $(\alpha, \beta)$ of (\ref{eq22}), 
with $\alpha \leq \beta$. 
Clearly, by {\rm (A4)},  any solution of 
\begin{equation}
\label{1}
- \Delta u = d(x)u + \mu_+ |\nabla u|^2 + h_+(x), \quad u \in X
\end{equation}
is an upper solution of (\ref{eq22}).  Moreover, 
a solution of
\begin{equation}
\label{222}
- \Delta u = d(x)u - \mu_- |\nabla u|^2 - h_-(x), \quad u \in X
\end{equation}
is a lower solution of (\ref{eq22}). 
Now if $w \in X$ is a solution of 
\begin{equation}
\label{3b}
- \Delta u = d(x)u + \mu_- |\nabla u|^2 + h_-(x), \quad u \in X,
\end{equation}
then $u = - w$ satisfies (\ref{222}). 
Thus if we find a non negative solution $u_1 \in X$ of (\ref{1}) and a non negative solution $u_2 \in X$ of (\ref{3b}) then, setting $\beta = u_1$ and $\alpha = -u_2$, we have the required couple of lower and upper solutions for Theorem \ref{sousBo}. By Proposition \ref{prop1}, we know that such non negative solutions of (\ref{1}) and (\ref{3b}) exist and this concludes the proof.
\end{proof}

As a direct consequence of the previous proposition, we obtain

\begin{cor}
\label{prop3}
Assume  {\rm (A1)} and, if $\mbox{\rm meas}(\Omega \backslash \mbox{\rm Supp} \, c) > 0$, assume also that {\rm (Hc)} holds. 
Then $(P_{\lambda})$ has a solution for any $\lambda <0$.
\end{cor}



As another direct consequence of Proposition \ref{prop2}, just noting that $W_d=H$ in case $d(x)\equiv 0$,
 we have

\begin{cor}
\label{prop3a}
Assume {\rm (A1)} and {\rm (H0)} hold.  Then $(P_0)$ has a  solution.
\end{cor}


\begin{remark}\label{ktL2}
Assume that 
$c$  and $h$ belong to $L^p(\Omega)$ for some $p > \frac{N}{2}$,
and that 	$\mu \in L^{\infty}(\Omega)$. 
	Assume that there exists an open subset $O(c)$ in $\Omega$ with $C^1$ boundary $\partial O(c)$  such that 
	$c(x)=0$ a.e. in $\overline{O(c)}$, $c(x)<0$ a.e. in $ \Omega\setminus O(c)$
and
		$\mu(x)\geq\mu_1>0$,   in  $\overline{O(c)}$. Then
%
$W_c=H_0^1(O(c))$
and
a necessary condition for the existence of a solution of $(P_\lambda)$ is that   
%
%
%
%
%
%
%
%
    \begin{equation}\label{k1}
    \inf_{\{\phi\in W_c, ||\phi||=1\}} \int_\Omega \left(\frac{1}{\mu(x)}|\nabla\phi|^2 -h(x)\phi^2 \right)\, dx >0.
    \end{equation}
Indeed, to show \eqref{k1}, we use an argument inspired by \cite{AbDaPe,FeMu1, FeMu2}.  
Suppose that $(P_\lambda)$ has a solution $u\in X$.  Then for any $\phi\in C_0^\infty(\Omega)$ we have
    \begin{equation}\label{k20}
    \int_\Omega\left(\nabla u\nabla(\phi^2)-\lambda c(x)u\phi^2-\mu(x)|\nabla u|^2\phi^2-h(x)\phi^2\right)\, dx = 0.
    \end{equation}
    and hence, for every $\phi\in C_0^\infty(\Omega)\cap W_c$ we obtain
    \begin{equation}\label{k2}
    \int_{O(c)}\left(\nabla u\nabla(\phi^2)-\mu(x)|\nabla u|^2\phi^2-h(x)\phi^2\right) dx = 0.
    \end{equation}
But, for  $\phi\in C_0^\infty(\Omega)\cap W_c$, by Young inequality,
    \begin{equation}\label{k3}
    \begin{array}{rcl}
		\displaystyle
		\int_{O(c)}\nabla u\nabla(\phi^2)\,dx &=& 
		\displaystyle
		\int_{O(c)} 2\phi\nabla u\nabla\phi\, dx
    \\
		&\leq& 
		\displaystyle
		\int_{O(c)} \left(\frac{1}{\mu(x)}|\nabla\phi|^2+\mu(x)|\nabla u|^2\phi^2\right) dx
    \end{array}
		\end{equation}
and thus by density
    $$  
		\int_{O(c)}\left(\frac{1}{\mu(x)}|\nabla\phi|^2 -h(x)\phi^2\right) dx \geq 0 \quad
        \mbox{for all }\phi\in W_c.
    $$

Thus, the infimum in (\ref{k1}) is non negative. If it is zero then, by Poincar\'e inequality, we also have that
\begin{equation}
    \label{k4}
    \inf_{\{\phi\in W_c: ||\phi||_2=1\}} \int_{O(c)}\left(\frac{1}{\mu(x)}|\nabla\phi|^2 -h(x)\phi^2\right) dx=0.
    \end{equation}
Let us show that it cannot take place.
Arguing by contradiction we assume that  \eqref{k4} hold.
Then, by standard arguments, there exists a $\phi_0\in W_c\setminus\{ 0\}$ such that
    \begin{equation}
    \label{k5}
    \int_{O(c)} \left(\frac{1}{\mu(x)}|\nabla\phi_0|^2 -h(x)\phi_0^2\right)\, dx=0.
    \end{equation}
    In addition, $\phi_0$ is an eigenfunction  associated to the first eigenvalue (which we are assuming equal to zero) of the elliptic eigenvalue problem
    $$
   \left\{  \begin{array}{c}
    -\mbox{div\,} \left( \displaystyle\frac{\nabla \phi}{\mu(x)}\right) - h(x) \phi = \lambda \phi \, , \ \ \mbox{in } O(c),
    \\
    \phi=0, \ \ \mbox{ on }\partial O(c) .
    \end{array}
    \right.
    $$
 As a consequence, we may assume that $\phi_0(x)>0$ in $O(c)$.
    
Setting $\phi=\phi_0$ in \eqref{k2}, we have by \eqref{k5} that
    $$  
    \int_{O(c)} \left(
    2 \phi_0  \nabla u\nabla\phi_0    
    -\mu(x)|\nabla u|^2\phi_0^2 -\frac{1}{\mu(x)}|\nabla\phi_0|^2\right)\, dx = 0.
    $$
That is,
    $$  \int_{O(c)} \Big |\frac{1}{\sqrt{\mu(x)}}\nabla\phi_0-\sqrt{\mu(x)}\phi_0\nabla u \Big|^2\, dx = 0
    $$
from which we deduce that $\nabla\phi_0=\mu(x)\phi_0\nabla u$ in $O(c)$,
 i.e.,
 \begin{equation}
\label{ajout1}
     \nabla u=\frac{1}{\mu(x)}\frac{\nabla\phi_0(x)}{\phi_0(x)}  \quad \mbox{in}\ O(c).
\end{equation}
  We also observe that $\frac{\partial\phi_0}{\partial\nu}(x)>0$ on $\partial O(c)$, where $\nu$ is
an inner normal vector at $\partial O(c)$ (see e.g. \cite[Lemma 3.4]{GiTu}).  This implies that
$\frac{\nabla\phi_0(x)}{\phi_0(x)}\not\in L^2(O(c))$ and then by  the fact that $\mu(x)\geq \mu_1$ in $O(c)$ and  (\ref{ajout1})  we deduce that $\nabla u\not\in L^2(\Omega)$.  This contradicts $u\in X$  proving 
that  \eqref{k4} is impossible and thus  \eqref{k1} holds.
\medbreak

Now if in addition to the above assumptions we assume that $\mu(x)\equiv \mu>0$ is a constant and $h(x)\geq 0$  it follows from \eqref{k1} that,  if $(P_\lambda)$
has a solution, we have
	\begin{equation}\label{k6}
	\inf_{\{\phi\in W_c: \|\phi\|=1\}} \int_\Omega\left(|\nabla\phi|^2-\mu h(x)\phi^2\right)\, dx>0.
	\end{equation}
Note that under these assumptions,
{\rm (Hc)} coincides with \eqref{k6} and thus $(P_\lambda)$ when $\lambda<0$ has 
a solution if and only if {\rm (Hc)} holds.  The statement in Remark \ref{ktL1} also follows in a similar
way.
Finally when $\lambda=0$ (equivalently when $c\equiv 0$), we have $O(c)=\Omega$, $W_c=H_0^1(\Omega)$ and \eqref{k6} reduces to {\rm (H0)}. Thus 
$(P_0)$ has a solution if and only if {\rm (H0)}
holds.
\end{remark}

\section{Uniqueness results}
\label{Sectionuniqueness-0}

As in the previous section, we consider the boundary value problem \eqref{31}.  Here we assume 
	$$	
	\hspace{1cm}
\left\{ \begin{array}{c} d \mbox{ and } h \mbox{ belong to }  L^p(\Omega) \quad \mbox{for some } p > \frac{N}{2},
\\[2mm]
		d(x)\leq 0 \ \mbox{in}\ \Omega\mbox{ and } \mu \in L^{\infty}(\Omega).  
\end{array}
\right.
\leqno{\mathbf{(A5)}} 
%
	$$
Our main result is
\begin{prop}\label{UUniqueness}
Assume that {\rm (A5)} hold. Then  \eqref{31} has at most one solution.
\end{prop}

To prove Proposition \ref{UUniqueness} we shall first prove that the solutions of \eqref{31} belong to 
$ C(\overline \Omega)\cap W^{1,N}_{loc}(\Omega)$. Then, 
using this additional regularity, we prove the uniqueness.

\begin{remark}
\label{lambda_0}
Proposition \ref{UUniqueness} implies that $(P_{\lambda})$ for $\lambda  \leq 0$ has at most one solution.
\end{remark}

\begin{remark}
\label{newrequire1}
As we mention in the Introduction a general theory of uniqueness for problems with quadratic growth in the gradient was developed in \cite{BaMu} and extended in \cite{BaBlGeKo}. The uniqueness results closer to our setting are Theorems 2.1 and 2.2 of \cite{BaBlGeKo}. Unfortunately it is not possible to use directly these results to derive Proposition \ref{UUniqueness}.  Indeed, since $d(x)$ may vanish on some part of $\Omega$, \cite[Theorem 2.1]{BaBlGeKo} is not applicable. Also, to use \cite[Theorem 2.2]{BaBlGeKo} which corresponds to the case $\lambda =0$, we need either $h(x)$ to have a sign or to be sufficiently small. 
\end{remark}

\begin{lem}
\label{betterreg} 
Assume that {\rm (A5)} hold.
Then any solution of \eqref{31} belongs to $C(\overline\Omega)\cap W^{1,N}_{loc}(\Omega)$.
\end{lem}
\begin{proof} Let $u \in X$ be an arbitrary solution of \eqref{31}. We divide the proof that $u \in W^{1,N}_{loc}(\Omega)$ into three steps.
\medbreak

\noindent {\bf Step 1}.  $u \in C(\overline{\Omega})$. 
\medskip

Since condition (A)  holds the result follows directly  
from 
\cite[Theorem IX.2.2]{LU68}. Indeed,  \eqref{31} is of the form of equation (1.1) of Section IV.1 of \cite{LU68}. In addition, under {\rm (A5)}  the assumptions (1.2)-(1.3) considered in \cite[Section IV.1]{LU68}  are satisfied.
 Hence, $u \in C^{0, \alpha}(\overline{\Omega})$ for some  $\alpha \in (0,1)$ and in particular $u\in C(\overline\Omega)$.
\medbreak

\noindent {\bf Step 2}. $u \in W^{1, q}_{loc}(\Omega)$ for some $q >2$. \medskip

Here we use \cite[Proposition 2.1, p.145]{Gi} or alternatively \cite[Theorem 2.5 p.155]{GiMo},  (see also \cite[Th\'eor\`eme 2.1]{BoMuPu2})  to deduce 
that $u \in W^{1,q}_{loc}(\Omega)$ for  some $q>2$. 
\medbreak

\noindent {\bf Step 3}. Conclusion. \medskip


We follows some arguments of \cite{BeFr, Fr}, see also \cite{DoGi}. First note that without restriction we can assume that $q <N$. Since $u \in W^{1,q}_{loc}(\Omega)$ we have,
\begin{equation}\label{t1}
 - \Delta u = \xi(x) \quad \mbox{where} \quad \xi(x) = d(x)u + \mu(x) |\nabla u|^2 + h(x) \in L^{\frac{q}{2}}_{loc}(\Omega).
\end{equation}
By standard regularity argument, see for example \cite[Theorem~9.11]{GiTu}, we deduce that $u \in W^{2,\frac{q}{2}}_{loc}(\Omega)$. Now using Miranda's interpolation Theorem \cite[Teorema IV]{Mi} between  $C^{0,\alpha}(\overline{\Omega})$ and $W^{2,\frac{q}{2}}_{loc}(\Omega)$  it follows, since $u \in C^{0,\alpha}(\overline{\Omega})$, that 
$$u \in   W^{1,t_1}_{loc}(\Omega)  \quad \mbox{where} \quad t_1 =  \frac{\frac{q}{2} (2- \alpha) - \alpha}{1 - \alpha} >q .$$ If $t_1 \geq N$ we are done. Otherwise from (\ref{t1}) and  classical regularity $u \in W^{2, \frac{t_1}{2}}_{loc}(\Omega)$. Denoting
\begin{equation}\label{t2}
t_n = \frac{\frac{t_{n-1}}{2} (2- \alpha) - \alpha}{1 - \alpha} > t_{n -1} > q > 2 
\end{equation}
by a bootstrap argument we get $u \in W^{2, \frac{t_n}{2}}_{loc}(\Omega)$ for all $n \in \N$ as long as $t_{n-1} \leq N$. We now claim that the sequence $\{t_n\}$ does not converge before reaching $N$. Indeed if we assume that $\{t_n\}$ has a finite limite $l$ we deduce from (\ref{t2}) that $l=2$ which contradicts $t_n >q >2$. At this point the proof of the lemma is completed.
\end{proof}

Using the fact that, under {\rm (A5)}, the solutions of \eqref{31} belong to 
$C(\overline\Omega)\cap W^{1,N}_{loc}(\Omega)$ we can now obtain our uniqueness result. Here we adapt an argument from \cite{BeMeMuPo}, based in turn on an original idea from \cite{BoMa}.

\begin{lem}
\label{Colette}
Assume that {\rm (A5)} hold. Then \eqref{31} has at most one solution in $X\cap W^{1,N}_{loc}(\Omega) \cap C(\overline\Omega)$.
\end{lem}
\begin{proof}

Let us assume the existence of two solutions $u_1$, $u_2$ of \eqref{31} in $X\cap W^{1,N}_{loc}(\Omega) \cap C(\overline\Omega)$. Then $v=u_1-u_2$ is a solution of
\begin{equation}
\label{2}
\begin{array}{c}
-\Delta v = \mu(x)(\nabla u_1 + \nabla u_2)\,\nabla v + d(x)v,
\quad  v\in X\cap W^{1,N}_{loc}(\Omega) \cap C(\overline\Omega).
\end{array}
\end{equation}
For every $c\in\mathbb R$, let us consider the set 
$\Omega_c=\{x\in \Omega \, :\, |v(x)|=c\}$ and 
$$
J=\{c\in \mathbb R  \, :\, \mbox{meas\,}\Omega_c>0\}.
$$
As $|\Omega|$ is finite, $J$ is at most countable and, since for all $c\in \mathbb R$, $\nabla v =0$ a.e. on $\Omega_c$,  we also have
\begin{equation}
				\label{mes}
\nabla v=0 \mbox{ a.e. in } \bigcup_{c\in J} \Omega_c.
\end{equation}
Define
$
Z= \Omega \setminus \bigcup_{c\in J} \Omega_c 
$ and let
 $G_k:\mathbb R\to \mathbb R$ be defined by
\begin{equation}
					\label{Gk}
G_k(s)= \left\{ \begin{array}{ll}0,&\mbox{if } |s|\leq k,
\\
(|s|-k)\,\mbox{sgn}(s),&\mbox{if }|s|  > k.
\end{array} \right.
\end{equation}
Now, using
 $\varphi= G_k (v)$ as test function in \eqref{2},
we  deduce  for all $k\geq 0$ that
\begin{eqnarray*}
\displaystyle
\|\nabla  G_k(v)\|_{2}^2 &=& 
				\int_{\Omega}|\nabla v|^2\chi_{\{|v|\geq k\}} \, dx
\\
&=&
\displaystyle
\int_{\Omega}\mu(x) (\nabla u_1+\nabla u_2)\, \nabla v\, G_k (v) \, dx + \int_{\Omega}d(x)\, v\, G_k(v) \, dx .
\end{eqnarray*}
%
Since  $v \in C(\overline{\Omega})$ we have that $G_k(v)$ has a compact support in $\Omega$  for all  $k >0$, which together to the fact that $d(x) \leq 0$ on $\Omega$ and \eqref{mes} implies that 
\arraycolsep1.5pt
\begin{equation}
\label{3bb}
\begin{array}{rcl}
\displaystyle
\|\nabla  G_k(v)\|_{2}^2
&\leq&
\displaystyle
\int_{\Omega}\mu(x) (\nabla u_1+\nabla u_2)\, \chi_{\{|v|\geq  k\}\cap Z}\, \nabla v\, G_k (v) \, dx
\\
&=&
\displaystyle
\int_{\Omega}\mu(x) (\nabla u_1+\nabla u_2)\, \chi_{\{|v|\geq  k\}\cap Z}\, \nabla  G_k (v)\,  G_k (v) \, dx
\\[4mm]
&\leq&
\displaystyle
\|\mu\|_{\infty} \|\nabla u_1+\nabla u_2\|_{L^N(\{|v|\geq k \}\cap Z)}\|\nabla  G_k (v)\|_{2} \| G_k (v)\|_{2^*}
\\[2mm]
&\leq&
\displaystyle
\mathcal{S}^{-1}_N\|\mu\|_{\infty} \|\nabla u_1+\nabla u_2\|_{L^N(\{|v|\geq  k \}\cap Z)}\|\nabla G_k (v)\|_{2}^2,
\end{array}
\end{equation}
\arraycolsep5pt
where  we recall that  $ \mathcal{S}_N$ denotes the Sobolev constant.
\medbreak

Assume by contradiction that $v\not\equiv 0$ and consider the function  $F:  ]0, \| v\|_\infty] \to\mathbb R$ defined by
$$
F(k)= \mathcal S^{-1}_N\|\mu\|_{\infty} \|\nabla u_1+\nabla u_2\|_{L^N(\{|v|\geq k\}\cap Z)}
, \ \ \forall \,  0<k<\| v\|_\infty.
$$
Observe that $F$ is non-increasing with  $F(\| v\|_\infty)=0$. Moreover, by definition of $Z$ we have that $F$ is continuous and we can choose $0<k_0<\| v\|_\infty$ such that
$F(k_0) < 1$. By \eqref{3bb}, $\|\nabla G_{k_0} (v)\|_{2}^2\leq F(k_0) \|\nabla  G_{k_0} (v)\|_{2}^2$, which implies that $\|\nabla G_{k_0} (v)\|_{2}=0$, i.e. $|v|\leq k_0< \| v\|_\infty$, a contradiction proving that necessarily $v=0$ and hence $u_1=u_2$ concluding the proof.
%
%
%
%
%
\end{proof}

\begin{proof}[Proof of Proposition \ref{UUniqueness}.]
This follows directly from  Lemmas \ref{betterreg} and \ref{Colette}.
\end{proof}

\section{Uniform $L^\infty$-estimates and existence of a continuum}
\label{Sectionuniqueness}

As in the previous section, we consider the boundary value problem \eqref{31} under the condition {\rm (A5)}.
\begin{lem}\label{4.2new}
Assume that {\rm (A5)} hold and that \eqref{31} has a solution $u_0\in X$.  Then 
\vspace{1mm}
\begin{enumerate}
\item[i)] For any 
	$\widetilde d(x)\in L^p(\Omega)$, $p>\frac{N}{2}$ with 
	$\widetilde d(x)\leq d(x)$, the problem
	\begin{equation}
	\label{n1}
	-\Delta u=\widetilde d(x)u+\mu(x)|\nabla u|^2 +h(x), \ u\in X
	\end{equation}
	has a unique solution $u\in X$.  Moreover,     
	$u$ satisfies
	$$	
	\|u\|_\infty \leq 2\|u_0\|_\infty.
	$$
\item[ii)] There exists $M_1>0$ such that for any $t\in [0,1]$ any solution
	$u_t$ of 	
	\begin{equation}
	\label{n4}
	-\Delta u = (d(x)-1)u +(1-t)\mu(x)|\nabla u|^2 +h(x), \quad t\in [0,1]
	\end{equation}
	 satisfies
	$\| u_t\|_\infty \leq M_1$.
\end{enumerate}
\end{lem}
\smallskip

\begin{proof} i)
Let $u_0\in X$ be a solution of \eqref{31} and set
	$$	
	\beta(x)=u_0(x)+\|u_0\|_\infty, \quad \alpha(x)=u_0(x)-\|u_0\|_\infty.
	$$
\vspace{1mm}
Then $\alpha\leq 0\leq \beta$  and, using that  $\widetilde d(x)\leq d(x)\leq 0$, we have
	\begin{eqnarray*}
	-\Delta \beta &=& d(x)(\beta-\|u_0\|_\infty)+\mu(x)|\nabla \beta|^2 +h(x)
	\\
	&=& \widetilde d(x)\beta +\mu(x)|\nabla \beta|^2 +h(x) +(d(x)-\widetilde d(x))\beta-d(x)\|u_0\|_\infty
	\\
	&\geq& \widetilde d(x)\beta +\mu(x)|\nabla \beta|^2 +h(x).
	\end{eqnarray*}
\vspace{1mm}
Thus $\beta$ is an upper solution of \eqref{n1}.  Similarly $\alpha$ is a lower solution of \eqref{n1}.
By Theorem \ref{sousBo}, \eqref{n1} has a solution $u(x)$ satisfying
	$$	
	\alpha(x) \leq u(x) \leq \beta(x) \quad \mbox{in}\ \Omega.
	$$
Since uniqueness of solutions of \eqref{n1} follows from Proposition \ref{UUniqueness}, this concludes the proof of the Point i). 
\medbreak

ii) Since $d(x)\leq 0$,   then  $\mbox{Supp}\, (d(x)-1) =\Omega$ and
thus, by Proposition \ref{prop1}, there exists a non negative solution $\beta$ (resp. $\alpha$) of
	$$	
	-\Delta u=(d(x)-1) u+\|\mu^+\|_\infty |\nabla u|^2+h^+
	$$
(resp.
	$	-\Delta u= (d(x)-1) u+\|\mu^-\|_\infty |\nabla u|^2+ h^-$).
	For any $t\in [0,1]$, we can observe that $\beta$ (resp. $-\alpha$) is an 
  upper (resp. lower) solution
of \eqref{n4}.  Thus there exists a solution $u_t$ of \eqref{n4} satisfying
$-\alpha\leq u_t\leq \beta$.  By Proposition \ref{UUniqueness}, uniqueness of solutions of \eqref{n4} holds
and thus case ii) holds with $M_1=\max(\|\beta\|_\infty, \|\alpha\|_\infty)$.
\end{proof}

We now transform \eqref{31} into a fixed point problem. By Corollary \ref{prop3} used with $c(x) \equiv 1$ and $\lambda=-1$, or alternatively  Theorem 2 of \cite{BoMuPu3}, we know that, for any $f \in L^p(\Omega)$ the problem
\begin{equation}
\label{pivot}
-\Delta u + u -\mu(x) |\nabla u|^{2} = f(x), \quad u \in X
\end{equation}
has a solution.
We also know from Proposition \ref{UUniqueness} that it is unique. Thus it is possible to define
the operator $K^\mu :L^p (\Omega )\longrightarrow X$ by  $K^\mu f=u$ where $u$ is the unique solution  of (\ref{pivot}). 
The following lemma,  which is proved in the Appendix, will be crucial.

\begin{lem}
\label{comp}
If $\mu\in L^{\infty}(\Omega)$ then the operator
$K^\mu$ is a completely continuous operator from $L^p (\Omega)$  into $C(\overline \Omega)$. 
\end{lem}

Next we define the continuous operator 
$N:C(\overline \Omega)\longrightarrow  L^p(\Omega )$  by, 
$$
N(u) =  (d(x) +1)\, u + h(x), \quad \mbox{ for any } u\in  C(\overline \Omega). 
$$
With these notations, $u\in  C(\overline \Omega)$ is a  solution of \eqref{31} if and only if
$u$ is a fixed point of $K^{\mu}\circ N$; i.e., if and only if 
 $$
u= K^{\mu}( N(u)).
$$

Now let $T : C(\overline \Omega) \to C(\overline \Omega)$ be given  by $T=K^{\mu}\circ N$. The following result hold.

\begin{prop}\label{4.1new}
Assume that {\rm (A5)}  holds and that \eqref{31} has a solution $u_0 \in X$.  Then
$$
	i(I-T,u_0)=1.
$$
\end{prop}

\begin{proof}
To show the proposition, we use homotopy arguments.  We consider  two  one-parameter problems, namely the problem \eqref{n4} with $t\in [0,1]$ and the following one  
\begin{equation}
\label{n3}
	-\Delta u = (d(x)-s)u +\mu(x)|\nabla u|^2 +h(x), \quad u\in X,	
\end{equation}
for $s\in [0,1]$.
Applying Lemma \ref{4.2new} we deduce that
\vspace{1mm}
\begin{enumerate}
\item Any solution
$u_s(x)$ of \eqref{n3} with $s\in [0,1]$ satisfies
	$	\| u_s\|_\infty \leq 2\|u_0\|_\infty 
	$. (Case  i) with $\widetilde d(x)=d(x)-s$).
	\vspace{1mm}
\item There exists $M_1>0$ such that for any $t\in [0,1]$ any solution
$u_t(x)$ of \eqref{n4} satisfies
	$	\| u_t\|_\infty \leq M_1
	$. (Case ii)).
\end{enumerate}
%
%
%
%
%
%
%
%
%
\vspace{1mm}
Observe that, if we set
	$$	\widetilde N_s(u)= (d(x)+1-s)u+h(x),
	$$
then  problem \eqref{n3} (resp. problem \eqref{n4}) is equivalent to $u-K^\mu(\widetilde N_s(u))=0$ (resp.
$u-K^{(1-t)\mu}(\widetilde N_1(u))=0$).  Thus setting $M=\max(2\|u_0\|_\infty, M_1)$, we have, for all $s$, $t\in [0,1]$ and all $u\in C(\overline\Omega)$ with $\|u\|_\infty=M$,
\vspace{1mm}
	$$	
	u-K^\mu(\widetilde N_s(u))\not=0, \quad u-K^{(1-t)\mu}(\widetilde N_1(u))\not=0.	
	$$
	\vspace{1mm}
Therefore, by homotopy invariance of the degree, we obtain
	\begin{eqnarray*}
	 \mbox{deg}(I-T,B(0,M),0) &=& \mbox{deg}(I-K^\mu\circ \widetilde N_0,B(0,M),0) 
	 \\
	&=& \mbox{deg}(I-K^\mu\circ \widetilde N_1,B(0,M),0) 
	\\
	&=& \mbox{deg}(I-K^0\circ \widetilde N_1,B(0,M),0) 
	=1.
	\end{eqnarray*}
By Proposition \ref{UUniqueness}, $u_0$ is the unique solution of \eqref{31} and thus
	$$	
	i(I-T,u_0)=\mbox{deg}(I-T,B(0,M),0) =1. 
	$$
\end{proof}



In the rest of the section, we apply the above results to the problem $(P_\lambda)$. First, 
from Lemma \ref{4.2new} we directly obtain the following a priori estimates for $(P_{\lambda})$ with $\lambda <0$.

\begin{cor}
\label{corA}
Assume {\rm (A1)} and, if $\mbox{\rm meas}(\Omega\setminus\mbox{\rm Supp}\, c)>0$, assume also that {\rm (Hc)} holds.
Then for any $\lambda_0<0$ there exists $R=R(\lambda_0)>0$ such that, for all $\lambda\leq \lambda_0$,  the unique solution
$u_\lambda$ of $(P_\lambda)$ satisfies
	$$	
	\|u_\lambda\|_\infty \leq R.
	$$
\end{cor}

\begin{proof}
The existence and uniqueness of solutions of $(P_\lambda)$ when $\lambda<0$, is already known from  Corollary \ref{prop3} and Proposition \ref{UUniqueness}.
Now the $L^\infty$-bound  is obtained from Lemma \ref{4.2new}, Point i) used with $d(x)=\lambda_0 c(x)$ and  $\widetilde d(x)
=\lambda c(x)$.  That is, the conclusion holds with $R(\lambda_0)=2\|u_{\lambda_0}\|_\infty$.
\end{proof}

\begin{remark}
A direct consequence of Corollary \ref{corA} is that none of $\lambda\in ]-\infty,0[$ is a bifurcation point from infinity of  $(P_\lambda)$.  (Recall that $\lambda \in \R$ is called a bifurcation point from infinity of $(P_\lambda)$
if there exists a sequence $\{ u_n\}$ of solutions of $(P_{\lambda_n})$ with $\lambda_n \to \lambda $  and $||u_n||_\infty\to\infty$).
\end{remark}

\section{Behaviour of the continuum in the half space $\{\lambda >0 \} \times C(\overline \Omega)$ }\label{Section2}

As a first consequence of  $\mathrm{(A2)}$ we obtain  the following result.

\begin{lem}
\label{l1}
Assume that $\mathrm{(A2)}$ holds. For $\gamma_1 >0$, the first eigenvalue of \eqref{eigenvaluep}, we have
\begin{enumerate}
\item[1)] If $\lambda <\gamma_1$,  any solution of problem $(P_{\lambda})$ is non negative. 
\smallskip
\item[2)] If $\lambda  =\gamma_1$, problem $(P_{\lambda})$  has no solution.
\smallskip
\item[3)] If $\lambda  > \gamma_1$, problem $(P_{\lambda})$  has no non negative solutions.
\end{enumerate}
\end{lem}

\begin{proof}
First we assume that  $\lambda <\gamma_1$. Let $u \in X$ be a solution of $(P_{\lambda})$. Using $u^-$ as test function in $(P_{\lambda})$  we obtain 
$$
 -  \int_{\Omega} (|\nabla u^-|^2 - \lambda c(x) |u^-|^2) dx  = \int_{\Omega} (\mu(x) |\nabla u|^2 u^- + h(x) u^-) dx.
$$
Since $\lambda <\gamma_1$ the left hand side is negative and since $\mu(x) \geq 0$ and $h(x) \geq 0$ the right hand side positive. So necessarily $u^- \equiv 0$ i.e., $u\geq 0$. This proves Point 1). \medskip

Now let $u\in X$ be a solution of 
$(P_{\lambda})$. Using $\varphi_1 >0$, the first eigenfunction of \eqref{eigenvaluep}, as test function in $(P_{\lambda})$ we obtain
\begin{eqnarray*}
(\gamma_{1} -\lambda ) \int_{\Omega}  c(x) u\varphi_{1} dx
&=& 
\int_{\Omega} \nabla u \nabla \varphi_{1} dx - \int_{\Omega} \lambda  c(x) u\varphi_{1}dx
\\
&=&
 \int_{\Omega} \mu(x) |\nabla u|^{2}\varphi_{1} dx + \int_{\Omega} h(x) \varphi_{1} dx .
\end{eqnarray*}
Since $\mu (x)\geq 0$ and $h(x)\gneqq 0$, the right hand-side of the above inequality is  positive. Thus when $\lambda = \gamma_1$, $(P_{\lambda})$ has no  solution and Point 2) is proved.

 Finally, when $\lambda > \gamma_1$ and $u\in H$ is a non negative solution of 
$(P_{\lambda})$, the left hand-side is non positive which contradicts the positivity of  the right hand side. This  proves Point 3).
\end{proof}



To prove the second part of Theorem \ref{th3}, the key point is the derivation of a priori bounds for solution of $(P_{\lambda})$ for $\lambda >0$. Actually we derive these bounds under a slightly more general assumption than needed. \medskip

We consider the problem
\begin{equation*}
\label{EE1}
-\Delta u = \lambda c(x)u+ H(x, \nabla u), \quad u \in X,
\eqno{(R_{\lambda})}
\end{equation*}
where we assume
$$
\leqno{\mathbf{(A6)}} \hspace{1cm}
\left\{ 
\begin{array}{c} c\gneqq 0  \mbox{ and } c  \mbox{ belong to }  L^p(\Omega) \quad \mbox{for some } p > \frac{N}{2}\\ \\
\mu_1 [|\xi|^2 +h(x)] \leq H(x, \xi) \leq \mu_2 [|\xi|^2 + h(x)] \\\\
     \mbox{ for some } 0 < \mu_1 \leq \mu_2 < \infty \mbox{ and } h \geq 0 \mbox{ with } h \in L^p(\Omega).  
   \end{array} 
\right.
$$

Adapting the approach of \cite{BT77}, we prove the following result: 

\begin{prop}
\label{bounds1}
Assume that $\mathrm{(A6)}$ holds.  Then for any $\Lambda_1>0$ there exists a constant $M >0$ such that, 
for each $\lambda\geq \Lambda_1$, any non negative solution $u$ of $(R_{\lambda})$ satisfies
$$
 ||u||_{\infty} \leq M.
$$
\end{prop}

\noindent In the proof of Proposition \ref{bounds1} the following two technical 
 lemmas will be useful.

\begin{lem}\label{tec1}
Let $p > \frac{N}{2}$ and $\theta \in\, ]0,1[$. There exist  $\alpha$, $r \in\, ]0,1[$ such that, if we define
    \begin{equation}\label{a0}
    q =1+ r + \frac{1+ \theta \alpha}{1- \alpha}, \quad \tau = \frac{1}{q} \, \frac{\alpha}{1- \alpha}
    \end{equation}
then it holds
\begin{equation}
\label{boundsq}
\frac{1}{p} \leq q \leq \frac{2N(p-1)}{p(N-2 + 2 \tau)}
\end{equation}
and
\begin{equation}
\label{boundsalpha}
1 - \alpha < \frac{2}{q}.
\end{equation}
\end{lem}

\begin{proof}
First observe that for all $\alpha \in \,]0,1[$, there exists $r_0>0$ such that, for any $0 < r \leq r_0$, (\ref{boundsalpha}) holds true. Indeed, since $r >0$, we have
$$
q > 1 + \frac{1+ \theta \alpha}{1 - \alpha} = \frac{2 - \alpha + \theta \alpha}{1 - \alpha} 
\quad \mbox{ or equivalently } \quad 
\frac{2}{q} < \frac{2(1- \alpha)}{2 - \alpha + \theta \alpha}.
$$
Also letting $r \to 0^+$ we obtain
$$ \frac{2}{q} \nearrow \frac{2(1- \alpha)}{2 - \alpha + \theta \alpha}.$$
Thus if 
\begin{equation}
\label{onalpha}
1 - \alpha < \frac{2(1- \alpha)}{2 - \alpha + \theta \alpha}
\end{equation}
there exists $r_0 >0$ such that, for all $0 < r \leq r_0$, (\ref{boundsalpha}) is satisfied. But (\ref{onalpha}) is equivalent to $\alpha (\theta-1) < 0$ which is always true. \medskip

\noindent Now, observe that, from the definition of $q$, we have $q \searrow 2$ as $r \searrow 0$ and $\alpha \searrow 0$. Finally, we see from the definition of $\tau$, that $\tau \searrow 0$ as $\alpha \searrow 0$. Thus as $\alpha \searrow 0$,
$$ 
\frac{2N(p-1)}{p(N-2 + 2 \tau)} \nearrow \frac{2N(p-1)}{p(N-2 )} >2,
$$
where the inequality is obtained using the assumption that $p > \frac{N}{2}.$ At this point it is clear that taking $r >0$ sufficiently close to $0$ and $\alpha >0$ sufficiently close to $0$, that (\ref{boundsq}) will also hold.
\end{proof}

%

\begin{lem}
\label{tec2}
Let $b \in L^p(\Omega)$ with $p > \frac{N}{2}$. For any $p$, $q\geq 1$ and $\tau\in [0,1]$ satisfying \eqref{boundsq}, there exists $C>0$ such that,
for all $w\in H$
    $$  
		\left\|\frac{b^{1/q} w}{\varphi_1^\tau}\right\|_q \leq C\| b\|_p \|\nabla w\|_2,
    $$
		where $\varphi_1 >0 $ denotes the first eigenfunction of \eqref{eigenvaluep}.
\end{lem}

\begin{proof}
For $p$, $q\geq 1$, $\tau\in [0,1]$ satisfying \eqref{boundsq}, define $s\geq 1$ by
    $$  
    \frac{1}{s}=\frac{1}{2}-\frac{1-\tau}{N}.
    $$
It follows from the second inequality of \eqref{boundsq} that
    $   \frac{1}{q}\geq (1-\frac{1}{p})^{-1}\frac{1}{s}
    $,
and this implies
    $$  
    \frac{1}{pq}\leq \frac{1}{q}-\frac{1}{s}.
    $$
From the first inequality of \eqref{boundsq}, we have $\frac{1}{pq}\leq 1$.  Thus there exists $\nu\geq 1$
such that
    $$  \frac{1}{pq}\leq \frac{1}{\nu}\leq \frac{1}{q}-\frac{1}{s}.
    $$
That is $\nu\geq 1$ satisfies
    $$  \frac{\nu}{q}\leq p \quad \mbox{and}\quad \frac{1}{q}\geq \frac{1}{\nu}+\frac{1}{s}.
    $$
Now by the Sobolev's embedding and \cite[Lemma 2.2]{BT77}, we have, for some constant $C>0$,
    $$
    \left\|\frac{b^{1/q} w}{\varphi_1^\tau}\right\|_q  \leq  C\| b^{1/q}\|_\nu \left\|\frac{w}{\varphi_1^\tau}\right\|_s 
        \leq  C'\|b\|_p^{1/q} \|\nabla w\|_2
    $$
    and the lemma is proved.
\end{proof}

\begin{proof}[Proof of Proposition \ref{bounds1}]
Fix  $\lambda>\Lambda_1$ and let $u\in X$ be a non negative solution of $(R_\lambda)$. 
By Points 2)-3) of Lemma \ref{l1} we deduce that $\lambda<\gamma_1$. Hence without loss of generality we suppose
$\Lambda_1<\gamma_1$ and $\lambda\in [\Lambda_1,\gamma_1]$.
\vspace{2mm}

\noindent We define
    $$  
    w_i(x)=\frac{1}{\mu_i}(e^{\mu_iu(x)}-1) \, \mbox{ and } \,  g_i(s)=\frac{1}{\mu_i}\ln(1+\mu_i s) \quad i=1,2.
    $$
Then we have
    \begin{eqnarray}
    u &=& g_1(w_1)=g_2(w_2),                \label{a1}
    \\
    e^{\mu_i u}&=&1+\mu_i w_i, \quad i=1,2. \label{a2}
    \end{eqnarray}
  \vspace{1mm}  
Direct calculations give us
    \begin{eqnarray*}
    -\Delta w_i &=& \lambda e^{\mu_i u}c(x)u + e^{\mu_i u}[H(x,\nabla u)-\mu_i|\nabla u|^2] 
    \\
        &=& \lambda(1+\mu_iw_i)c(x)g_i(w_i) + (1+\mu_iw_i)[H(x,\nabla u)-\mu_i|\nabla u|^2].
    \end{eqnarray*}
    
\noindent Since $\Lambda_1\leq\lambda\leq \gamma_1$, we have by $\mathrm{(A6)}$

 \begin{eqnarray*}
    -\Delta w_1 &\geq& \Lambda_1(1+\mu_1w_1)c(x)g_1(w_1) +\mu_1(1+\mu_1w_1)h(x),   
    \\
    -\Delta w_2 &\leq& \gamma_1(1+\mu_2w_2)c(x)g_2(w_2)  +\mu_2(1+\mu_2w_2)h(x). 
    \end{eqnarray*}
    
\noindent Setting $A_1=\min(\Lambda_1,\mu_1)$, $A_2=\max(\gamma_1,\mu_2)$, it becomes

    \begin{eqnarray}
    -\Delta w_1 &\geq& A_1(1+\mu_1w_1)[c(x)g_1(w_1)+h(x)], \label{a3}
    \\
    -\Delta w_2 &\leq& A_2(1+\mu_2w_2)[c(x)g_2(w_2)+h(x)]. \label{a4}
    \end{eqnarray}
    
From the inequalities \eqref{a3} and \eqref{a4}, we shall deduce that $w_2$ is uniformly bounded in $H$.
This will lead to the proof of the theorem by classical results relating the $L^{\infty}$ norm of a  lower solution to its $H$ norm. 
We divide the proof into three steps.

\medskip

\noindent
{\bf Step 1.}
{\it Let $\theta={(\mu_2-\mu_1)\mu_2^{-1}}\in\, ]0,1[$.  Then there exists $C>0$ independent of 
$\lambda\in [\Lambda_1,\gamma_1]$ such that}
    \begin{eqnarray}
    &&\int_\Omega (1+\mu_1w_1)[c(x)g_1(w_1)+h(x)]\varphi_1\, dx \leq C,            \label{a5}\\
    &&\int_\Omega (1+\mu_2w_2)^{1-\theta}[c(x)g_2(w_2)+h(x)]\varphi_1\, dx \leq C. \label{a6}
    \end{eqnarray}

\noindent Indeed, using $\varphi_1 >0$ as a test function in \eqref{a3}, we have

    $$  
    \gamma_1\int_\Omega c(x)w_1\varphi_1\, dx 
        \geq A_1\int_\Omega (1+\mu_1w_1)[c(x)g_1(w_1)+h(x)]\varphi_1\, dx.
    $$
    
\noindent We note that for any $\varepsilon>0$ there exists $C_\varepsilon>0$ such that
$t\leq \varepsilon (1+\mu_1t)g_1(t)+C_\varepsilon$ for all $t\geq 0$.
Thus  
    $$  
    \gamma_1\int_\Omega c(x)w_1\varphi_1\, dx 
        \leq \varepsilon \gamma_1 \int_\Omega (1+\mu_1w_1)[c(x)g_1(w_1)+h(x)]\varphi_1\, dx +C_\varepsilon'
    $$
    
\noindent and choosing $\varepsilon=\frac{A_1}{2\gamma_1}$, we obtain \eqref{a5}. Now observe that by \eqref{a2}, 
    $$  1+\mu_1w_1= e^{\mu_1u}=(e^{\mu_2 u})^{1-\theta} = (1+\mu_2w_2)^{1-\theta}.
    $$
Thus from \eqref{a1} we see that \eqref{a6} is nothing but \eqref{a5}.

\medskip

\noindent
{\bf Step 2.}
{\it There exists a constant $C>0$ independent of $\lambda\in [\Lambda_1,\gamma_1]$ such that}
    \begin{equation}\label{unifestimates}
    \|\nabla w_2\|_2 \leq C.
    \end{equation}

First we use Lemma \ref{tec1} to choose $\alpha$, $r\in\, ]0,1[$ such that $q$ and $\tau$ given in \eqref{a0}
satisfy \eqref{boundsq} and \eqref{boundsalpha}.
\vspace{2mm}

Using $w_2$ as a test function in \eqref{a4} it follows that
    $$
    \|\nabla w_2\|_2^2
    \leq A_2\int_\Omega(1+\mu_2w_2)[c(x)g_2(w_2)+h(x)]w_2\, dx.
	$$
Now using H\"older's inequality, \eqref{a6} and since $ w_2\leq { (1+\mu_2w_2) \mu_2^{-1}}$ we have
    \begin{eqnarray*}
    \|\nabla w_2\|_2^2 
	&\leq& \frac{A_2}{\mu_2}\int_\Omega (1+\mu_2w_2)[c(x)g_2(w_2)+h(x)]\frac{\varphi_1^\alpha}{(1+\mu_2w_2)^{\theta\alpha}}
                \frac{(1+\mu_2w_2)^{1+\theta\alpha}}{\varphi_1^\alpha}\, dx
                \\
    &\leq& \frac{A_2}{\mu_2}\left(\int_\Omega (1+\mu_2w_2)[c(x)g_2(w_2)+h(x)]\frac{\varphi_1}{(1+\mu_2w_2)^\theta}\,dx\right)^\alpha 
    \\
    && \qquad \times
        \left(\int_\Omega (1+\mu_2w_2)[c(x)g_2(w_2)+h(x)]\frac{(1+\mu_2w_2)^{\frac{1+\theta\alpha}{1-\alpha}}}{\varphi_1^{\frac{\alpha}{1-\alpha}}}\,dx\right)^{1-\alpha} 
        \\
    &\leq& \frac{A_2}{\mu_2} C^\alpha 
    \left(\int_\Omega (1+\mu_2w_2)[c(x)g_2(w_2)+h(x)]\frac{(1+\mu_2w_2)^{\frac{1+\theta\alpha}{1-\alpha}}}{\varphi_1^{\frac{\alpha}{1-\alpha}}}\,dx\right)^{1-\alpha}.
    \end{eqnarray*}
We note that for  $r$ given by Lemma~\ref{tec1}, there exists 
$C_r>0$
    $$  g_2(t) \leq t^r +C_r \quad \mbox{for all}\ t\geq 0.
    $$
Thus, direct calculations shows that

	$$	(1+\mu_2w_2)[c(x)g(w_2)+h(x)](1+\mu_2w_2)^{\frac{1+\theta\alpha}{1-\alpha}} 
		\leq (c(x)+h(x))(w_2^q+C),
	$$
	
\noindent	where $q$ 
is given in \eqref{a0}.
Therefore for some $C$, $C'>0$ independent of $\lambda\in [\Lambda_1,\gamma_1]$
    $$   \|\nabla w_2\|_2^2
		\leq C\left(\int_\Omega \left(\frac{(c(x)+h(x))^{1/q}w_2}{\varphi_1^\tau}\right)^q\,dx\right)^{1-\alpha}+C',
    $$
    
 \noindent with $q$ and $\tau$ given in \eqref{a0}. Applying Lemma \ref{tec2}, we then obtain
    
    $$   \|\nabla w_2\|_2^2 \leq C\|c+h\|_p^{q(1-\alpha)}\|\nabla w_2\|_2^{q(1-\alpha)} +C'.$$
    
\noindent By \eqref{boundsalpha}, we have $q(1-\alpha)<2$ and this concludes the proof of Step 2.

\medskip

\noindent 
{\bf Step 3.} {\it Conclusion.}
\medskip

We just have to show that the uniform estimate \eqref{unifestimates} derived in Step 2 
gives an uniform  estimate in the  $L^{\infty}$ norm. 
Recall that, as a consequence of 
Theorem 4.1 of \cite{Tr}
combined with Remark 1 on page 289 of that paper (see also Remark 2 p. 202 of \cite{LU68}), we know that if $w\in X$ satisfies
$$
\begin{array}{cc}
-\Delta w\leq d(x) w + f(x),&\mbox{in }\Omega,
\\
w\leq 0,&\mbox{on }\partial\Omega,
\end{array}
$$
with $d$, $f\in L^{p_1}(\Omega)$ for some $p_1>\frac{N}{2}$, then $w$ satisfies
$$
\|w^+\|_{\infty}\leq C(\|w\|_{1}+\|f\|_{p_1}),
$$
where $C$ depends on $p_1$, $\mbox{meas}(\Omega)$ and $\|d\|_{p_1}$.

Since $w_2$ satisfies \eqref{a4},
we  apply the result of \cite{Tr} with 
$$d(x)=c(x)A_2 (1+\mu_2 w_2(x))\frac{\ln(1+\mu_2 w_2(x))}{\mu_2 w_2(x)}+A_2^2 h(x) \quad \mbox{and} \quad  f(x)=A_2 h(x).$$
Observe that, for any $r\in \,]0,1[$, there
exists $C >0$ such that, for all $x\in \Omega$,
$$
c(x)A_2 (1+\mu_2 w_2(x))\frac{\ln(1+\mu_2 w_2(x))}{\mu_2 w_2(x)}\leq C\, c(x)|w_2(x)|^r.
$$
Thus, since $c(x) \in L^p(\Omega)$ with $p > \frac{N}{2}$ and  $w_2$ is bounded in $L^{\frac{2N}{N-2}}(\Omega)$,  taking
$r >0$ sufficiently small  we see, using  H\"older's inequality, that  $c(x) |w_2(x)|^r \in L^{p_1}(\Omega)$ for some $p_1 > \frac{N}{2}$. Now as 
$h\in L^p(\Omega)$ for some $p>\frac{N}{2}$, clearly all the assumptions of Theorem 4.1 of \cite{Tr} are satisfied.  From (\ref{unifestimates}) we then deduce that there exists a constant $C>0$, 
independent of $\lambda \in [\Lambda_1, \gamma_1]$ such that
$$||w_2||_{\infty} \leq C.$$
Now since $u = g_2(w^2)$ we deduce that a similar estimate holds for the non negative solutions of $(R_{\lambda})$  and the proof of the proposition is completed.
\end{proof}

\section{Proofs of the main results.}\label{Proofs}

In this section we give the proofs of our three theorems.

\begin{proof}[Proof of Theorem \ref{negativevalue}]
The uniqueness of the solution of $(P_\lambda)$ for $\lambda\leq 0$ is a consequence of  Remark~\ref{lambda_0}. 
By  Corollary~\ref{prop3}, $(P_\lambda)$
with $\lambda<0$ has a  solution $u_\lambda\in X$. 
This proves Point 1). To establish the existence of a continuum of solutions of $(P_\lambda)$, we 
define $T_{\lambda} : C(\overline{\Omega}) \to C(\overline{\Omega})$ as

$$
T_\lambda(u)=K^\mu((\lambda c(x)+1)u+h(x)).
$$

\noindent Hence, $(P_\lambda)$ is transformed into the fixed point problem $u=T_\lambda(u)$.
From Proposition \ref{4.1new} we immediately deduce that,  for any  $\lambda<0$,

 $$
 i(I-T_{\lambda},u_{\lambda})=1.
 $$
 
\noindent Therefore, if we fix a $\lambda_0<0$, by Theorem \ref{1noacot} where $E=C(\overline\Omega)$ and
$\Phi(\lambda,u)=u-T_\lambda(u)$, there exists
a continuum $C=C^+\cup C^-$ of solutions of $(P_\lambda)$ emanating from $(\lambda_0,u_{\lambda_0})$.
Taking into account the unboundedness of $C^+$ and $C^-$ and
  Corollary \ref{corA}, necessarily 
$]-\infty,0[ \,\subset \mbox{Proj}_{\mathbb{R}}C$ and the proof of Point 2) is concluded.
\vspace{1mm}

To prove Point 3), we apply
Lemma \ref{4.2new} with $d(x)= \overline\lambda c(x)$, $\widetilde{d}(x)= \lambda c(x)$ and $ \lambda\leq \overline\lambda< 0$,  to deduce that

	$$	
	\|u_{\lambda}\|_\infty \leq 2\|u_{\overline\lambda}\|_\infty \quad \mbox{for all}\ \lambda\leq \overline\lambda< 0.
	$$
	
In particular, if
$C_0:=\liminf_{\lambda\to 0^-}\|u_{\lambda}\|_\infty
<\infty$, then there exists a sequence $\overline\lambda_n \to 0^-$ such that $C_0=\lim_{n\to\infty}\|u_{\overline\lambda_n}\|_\infty
<\infty$. Hence,   for every sequence $\lambda_n \to 0^-$ we deduce by the above inequality that
$\limsup_{n\to\infty}\|u_{\lambda_n}\|_\infty \leq 2C_0$, which implies that $\limsup_{\lambda\to 0^-}\|u_\lambda\|_\infty<\infty$.
%
%
Therefore, we have
either $\lim_{\lambda\to 0^-}\|u_\lambda\|_\infty=\infty$ or $\limsup_{\lambda\to 0^-}\|u_\lambda\|_\infty<\infty$. 
\vspace{1mm}

In the first case, 
using Lemma \ref{4.2new} with $d(x) \equiv 0$ and $\widetilde{d}(x) = \lambda c(x)$, we see that $(P_0)$  cannot have a solution. On the other hand, in 
 the last case, for any sequence $\lambda_n\to 0^-$, $(u_{\lambda_n})$ is a bounded
sequence in $L^\infty(\Omega)$.  Thus by Lemma \ref{comp},
	$$	
	u_{\lambda_n}=K^\mu((\lambda_n c(x)+1)u_{\lambda_n}+h(x))
	$$
	
\noindent
is relatively compact in $C(\overline\Omega)$.  Taking a subsequence if necessary, we may assume
$u_{\lambda_n}\to u_0$ in $L^\infty(\Omega)$ for some $u_0 \in X$.  It is clear that $u_0$ satisfies $u_0=K^\mu(u_0 + h(x))$,
that is, $u_0$ is a solution of $(P_0)$.  Since we have uniqueness of solutions of $(P_0)$ by 
Remark~\ref{lambda_0}, the limit $u_0$ does not depend on the choice of $\lambda_n$ and thus
we have $u_\lambda\to u_0$ in $L^\infty(\Omega)$ as $\lambda\to 0^-$. This ends the proof.
\end{proof}

\vspace{1mm}




\begin{proof}[Proof of Theorem \ref{th1.1a}]
If we assume that $(P_0)$ has a solution $u_0$ then using Lemma \ref{4.2new} with $d(x) \equiv 0$ and 
$\widetilde{d}(x) = \lambda c(x)$ we obtain the existence of a solution $u_\lambda$ of $(P_\lambda)$ for any $\lambda<0$. Using Remark \ref{lambda_0}
Point 1) follows.
\vspace{1mm}

Now by Proposition \ref{4.1new}, we know that $i(I-T_0,u_0)=1$.  Thus by Theorem \ref{1noacot}
there exists a continuum $C\subset\Sigma$ such that both 
	$$	
	C\cap([0,\infty[\,\times C(\overline\Omega)) \quad \mbox{ and }\quad   C\cap(]-\infty,0]\,\times C(\overline\Omega))
	$$
are unbounded.  Clearly $\{ (\lambda,u_\lambda):\ \lambda\in\, ]-\infty,0]\,\}\subset C$ and 
Point 2)  holds.
\end{proof}

\vspace{1mm}

\begin{proof}[Proof of  Theorem \ref{th3}.]
Let $C \subset \Sigma$ be the continuum obtained in Theorem \ref{th1.1a}. By Lemma \ref{l1}, Point 2) we know that  
$]-\infty,0]\subset \, \mbox{\rm Proj}_{\R} C \subset \,]-\infty, \gamma_1[$. Lemma \ref{l1}, Point 1) shows that it consists of non negative functions. 
In addition,  by Theorem \ref{th1.1a}, Point 2),
$C\cap ([0,\gamma_1 [\times C(\overline{\Omega}))$ is unbounded    and hence its projection  on $C(\overline\Omega)$ has to be unbounded. 
%
%
%
%
Now we know, by Proposition  \ref{bounds1}, that for every $\Lambda_1 \in\, ]0,\gamma_1[$,  there is  an a priori bound on the non negative solutions for $\lambda\geq  \Lambda_1$. This means that the projection of $C\cap ( [\Lambda_1, \gamma_1[ \times C(\overline\Omega))$ on $C(\overline\Omega)$ is bounded.  
Thus $C$ must emanate from infinity to the right of $\lambda=0$. 
This proves the first part of the theorem.

\medbreak
 Since $C$ contains $(0, u_0)$ with $u_0$ the unique solution of $(P_0)$, there exists a $\lambda_0 \in\, ]0,\gamma_1[$ such that the problem $(P_{\lambda})$ has at least two solutions for $\lambda\in \,]0,\lambda_0[$. 
 At this point the proof of the theorem is completed.
\end{proof}
%

 

\section{Appendix : Proof of Lemma \ref{comp}.} \label{appendix}

To prove Lemma \ref{comp}, we need some preliminary results.

\begin{lem}
\label{estimationsinfini}
Let $\{f_n\} \subset L^p(\Omega)$ be a bounded sequence. Then the sequence $\{u_n\} = \{K^{\mu}(f_n)\}$ is bounded in $L^{\infty}(\Omega)$ and in $H.$
\end{lem}

\begin{proof}
First we observe that the boundedness of $\{u_n\}$  in $L^{\infty}(\Omega)$ is a direct consequence of Theorem 1 of \cite{BoMuPu3}. To show that $\{u_n\}$ is also bounded in $H$ we use a trick that can be found for example in  \cite{BoMuPu1,BoMuPu2}. 
Let $t = ||\mu||_{\infty}^2 /2$, $E_n = \exp(tu_n^2)$ and consider the functions $v_n = E_nu_n$. We have $v_n \in X$ and 
 $$
 \nabla v_n= E_n (1+2 t u_n^2) \nabla u_n.
 $$
Hence using $v_n$ as test functions in  
 $$
-\Delta u_n + u_n  = \mu(x) |\nabla u_n|^{2} + f_n(x), \quad  u_n\in X,
$$
and the bound of $\{u_n\}$ in $L^{\infty}(\Omega)$,
we obtain the existence of a constant $D>0$ such that 
\arraycolsep1.5pt
$$
\begin{array}{rcl}
\displaystyle
\int_{\Omega}E_n (1 &+& 2t u_n^2) |\nabla u_n|^2 dx 
+ 
\displaystyle
  \int_{\Omega}E_n u_n^2 dx 
\\[3mm]
  &=& 
  \displaystyle
\int_{\Omega} f_n(x) E_nu_n dx + \int_{\Omega}\mu(x) |\nabla u_n|^2 E_n u_n dx 
\\[3mm]
& \leq & 
\displaystyle
D 
+ ||\mu||_{\infty}\int_{\Omega} E_n^{1/2}|\nabla u_n| |u_n| |\nabla u_n| E_n^{1/2} dx 
\\[3mm]
& \leq & \displaystyle
D 
+ ||\mu||_{\infty} \left[ \frac{1}{2||\mu||_{\infty}} \int_{\Omega}E_n|\nabla u_n|^2 dx + \frac{1}{2} ||\mu||_{\infty} \int_{\Omega}u_n^2 |\nabla u_n|^2 E_n dx \right] 
\\[3mm]
& \leq & 
\displaystyle
D 
+ \frac{1}{2} \int_{\Omega}E_n(1+ 2t u_n^2) |\nabla u_n|^2 dx.
\end{array}
$$
\arraycolsep5pt
We then deduce that 
$$
\int_{\Omega}E_n  |\nabla u_n|^2 dx  + \int_{\Omega}E_n u_n^2 dx 
 \leq  2 D.
$$
Recording that $E_n \geq 1$, this shows that $\{u_n\}$ is bounded in $H$.
\end{proof}

\begin{proof}[Proof of Lemma \ref{comp}]
The proof we give is inspired by \cite{BoMuPu3} combined  with \cite[Remark 2.7]{ArCaLeMAOrPe} (based in turn on ideas from \cite{LU68}). 
\medskip

\noindent {\bf Step 1.} {\it $K^\mu$ is a bounded operator from $L^p(\Omega)$ to $C^{0,\alpha}(\overline\Omega)$
for some $\alpha\in\, ]0,1[$} \medskip 

Assume that $\{ f_n\}$ is a bounded sequence in $L^p(\Omega)$. By Lemma \ref{estimationsinfini},
$u_n=K^\mu (f_n) $ is bounded in $L^\infty (\Omega)$. We claim that $u_n$ is also bounded in $C^{0,\alpha}(\overline \Omega)$ for some $\alpha \in \,]0,1[$. Indeed,
consider a function $\zeta\in
C^\infty (\Omega) $ with $0\leq \zeta(x)\leq 1$, and compact support in a ball $B_\rho$ of radius
$\rho>0$, and set $A_{k,\rho}=\{x\in B_\rho\cap \Omega: |u(x)|>k\}$.
 
Let us  consider the function $G_k$ given by \eqref{Gk}. 
For $\varphi (s)=s e^{\gamma s^{2}}$ with $\gamma >0$ large (to be precised later)  we take 
$\phi=\varphi (G_k(u_n))\zeta^2$ as test function in (\ref{pivot}). Hence we have
\begin{eqnarray*}
\int_{\Omega} \nabla u_n \nabla (G_k(u_n)) \varphi '(G_k(u_n)) \zeta^2 dx
&=& \int_{\Omega} [- u_n + f_n(x) ] \varphi (G_k(u_n))\zeta^2 dx
\\
&&\hspace{10mm}+
\int_{\Omega} \mu(x) |\nabla u_n |^2 \varphi (G_k(u_n))\zeta^2 dx
\\
&&\hspace{20mm}-
2\int_{\Omega} \zeta \varphi (G_k(u_n)) \nabla u_n \nabla \zeta dx.
\end{eqnarray*}
Now observe that, for $\gamma>\frac{\|\mu\|_{\infty}^2}{4}$, we have $1+2\gamma s^2-\|\mu\|_{\infty} |s|\geq 1/2$ and hence $\varphi'(s)-\|\mu\|_{\infty} |\varphi(s)|\geq \frac12 e^{\gamma s^2}\geq \frac12$. Moreover, we have $G_k(u_n(x)) \zeta^2(x)=0$
for $x\not\in A_{k,\rho}$ and $\nabla G_k(u_n)=\nabla u_n$ in $A_{k,\rho}$. This implies that
\arraycolsep1.5pt
$$
\begin{array}{rl}
\displaystyle   \frac{1}{2}
    \int_{A_{k,\rho}}\!\!&\!\!
    |\nabla G_k(u_{n})|^{2} \zeta^2 dx
\\[3mm]
    &\leq
\displaystyle   \int_{A_{k,\rho}}  \left[ \varphi'(G_k(u_{n})) - \|\mu\|_{\infty} |\varphi (G_k(u_{n})) |\right] |\nabla G_k(u_{n})|^{2} \zeta^2 dx
\\[3mm]
    &\leq
\displaystyle   
\int_{A_{k,\rho}}  \!\! [- u_n + f_n(x)] \varphi (G_k(u_n))\zeta^2 dx
\\[3mm]
&\displaystyle
\hspace{25mm}
+
\int_{A_{k,\rho}} \!\!(|\mu(x)|-\|\mu\|_{\infty}) |\nabla u_n |^2 |\varphi (G_k(u_n))|\zeta^2
\\[3mm]
&\displaystyle
\hspace{65mm}
-
2\int_{A_{k,\rho}} \!\!  \zeta \varphi (G_k(u_n)) \nabla u_n \nabla \zeta  dx
\\[3mm]
&\leq
\displaystyle   
\int_{A_{k,\rho}}  \!\! [- u_n + f_n(x)] \varphi (G_k(u_n))\zeta^2 dx
+
2\int_{A_{k,\rho}}\!\!   |\zeta| \, |\varphi (G_k(u_n))| |\nabla u_n| \,|\nabla \zeta| dx.
\end{array}
$$
Now recall the existence of $C_1$ and $C_2$ such that, for all $n\in \mathbb N$, $\|u_n\|_{\infty}\leq C_1$ and $\|f_n\|_p \leq C_2$. Let $C_3$ such that, for all $s\in [-C_1,C_1]$, $|\varphi(s)|\leq C_3 |s|$ and recall that $0\leq \zeta\leq 1$. Hence we obtain $C=C(C_1,C_2,C_3)$ such that
\arraycolsep1.5pt
$$ 
\begin{array}{rcl}
\displaystyle   
\frac{1}{2}
    \int_{A_{k,\rho}}
    |\nabla G_k(u_{n})|^{2} \zeta^2 dx
    &\leq&
    \displaystyle C (\mbox{meas}(A_{k,\rho}))^{1-\frac1p}+ 2 C_3 \int_{A_{k,\rho}} |\zeta | |\nabla u_n| |\nabla \zeta|  |G_k(u_n)| dx
        \\
        &\leq&
    \displaystyle C (\mbox{meas}(A_{k,\rho}))^{1-\frac1p}+ \frac14 
\int_{A_{k,\rho}} |\zeta |^2 |\nabla u_n|^2 dx
    \\
        &&
        \hfill
    \displaystyle
    +4 C_3^2 \int_{A_{k,\rho}} |\nabla \zeta|^2  |G_k(u_n)|^2 dx,
\end{array}
$$
by using Young's inequality. Hence, recalling that, on 
$A_{k,\rho}$, we have $\nabla G_k(u_{n})=\nabla u_n$, we conclude that
$$
\begin{array}{rcl}
\displaystyle   
\frac14 \int_{A_{k,\rho}}
    |\nabla u_{n}|^{2} \zeta^2 dx
    &\leq&
    \displaystyle C \left((\mbox{meas}(A_{k,\rho}))^{1-\frac1p}+  \int_{A_{k,\rho}}  |\nabla \zeta|^2  |G_k(u_n)|^2 dx\right),
\end{array}
$$
\arraycolsep5pt
where $C=C(C_1,C_2,C_3)$ is a generic constant.\smallskip

Now we argue as in \cite[Theorem IV-1.1, p.251]{LU68}. For $\sigma\in\,]0,1[$, choose $\zeta$ such that $\zeta\equiv 1$ in the concentric ball $B_{\rho-\sigma\rho}$ (concentric to $B_{\rho}$) of radius $\rho-\sigma\rho$ and such that $|\nabla \zeta|< \frac{2}{\sigma\rho}$. Hence, we obtain
\arraycolsep1.5pt
$$
\begin{array}{rcl}
\displaystyle   
\int_{A_{k,\rho-\sigma\rho}}
    |\nabla u_{n}|^{2} dx
    &\leq&
    \displaystyle C (1+  (\max_{A_{k,\rho}}(|u(x)|-k))^2  \||\nabla \zeta|^2\|_{L^p(A_{k,\rho})})(\mbox{meas}(A_{k,\rho}))^{1-\frac1p}
    \\
    &\leq&
    \displaystyle C (1+ \frac{4}{\rho^2\sigma^2}(\rho^N\omega_N)^{1/p} (\max_{A_{k,\rho}}(|u(x)|-k))^2 )(\mbox{meas}(A_{k,\rho}))^{1-\frac1p},
\end{array}
$$
\arraycolsep5pt
where $\omega_N$ denotes the measure of the unit ball of $\mathbb R^N$. Hence, for $k\geq C_1\geq \max_{B_{\rho}}|u_n|-\delta$, we have
$$
\begin{array}{rcl}
\displaystyle   
\int_{A_{k,\rho-\sigma\rho}}
    |\nabla u_{n}|^{2} dx
    &\leq&
    \displaystyle \gamma \left(1+ \frac{1}{\sigma^2\rho^{2(1-\frac{N}{2p})}} (\max_{A_{k,\rho}}(|u(x)|-k))^2 \right)
    (\mbox{meas}(A_{k,\rho}))^{1-\frac1p}.
\end{array}
$$
This means that,  for $\delta >0$ small enough and  every $M\geq C_1\geq \|u_n\|_{\infty}$,
we have $u_n\in B_2(\Omega, M, \gamma, \delta, \frac{1}{2p})$ (see \cite[pag. 81]{LU68}).

Applying \cite[Theorem II-6.1 and Theorem II-7.1, p.90 and 91]{LU68}, we deduce that $u_n\in C^{0,\alpha}(\overline \Omega)$ with 
$\| u_n\|_{C^{0,\alpha}}$ bounded by a constant $C_4$ which depends only on  $\Omega, M, \gamma, \delta$ and the claim is proved.
\medbreak

\noindent {\bf Step 2.} {\it $K^\mu$ maps bounded sets of $L^p(\Omega)$ to relatively compact sets of $C(\overline\Omega)$.}
\medskip 

This can be easily deduced from Step 1 and  the compact embedding of  $C^{0,\alpha}(\overline \Omega)$ into 
$C(\overline \Omega)$.
\medskip

\noindent {\bf Step 3.} {\it $K^{\mu}$ is continuous from $L^p(\Omega)$ to $H$. } \medskip 

Let $\{f_n\} \subset L^p(\Omega)$ be  a sequence such that $f_n \to f$ in $L^p(\Omega)$ and let $\{u_n\}$ be the corresponding solutions of (\ref{pivot}). By Lemma \ref{estimationsinfini},  there exists $C>0$ such that, for all $n \in \N$, $||u_n||_{\infty} \leq C$ and $||u_n|| \leq C.$ Hence for every subsequence $\{u_{n_k}\}$, there exists a subsubsequence $\{u_{n_{k_j}}\} \subset X$  and $u \in X$ such that $u_{n_{k_j}} \rightharpoonup u$ weakly in $H$, $u_{n_{k_j}} \to u$ strongly in $L^{p'}(\Omega)$ and $u_{n_{k_j}} \to u$ almost everywhere. \smallskip

Let us prove that $u_{n_{k_j}} \to u$ strongly in $H$ and that $u$ is the solution of (\ref{pivot}). In that case we shall deduce that $u_n \to u$ in $H$, namely the continuity of $K^{\mu}$ from $L^p(\Omega)$ to $H$. Let us define $\tilde{u}_j = u_{n_{k_j}} -u$. Observe that $\tilde{u}_j$ satisfies
\begin{equation*}
-\Delta \tilde{u}_j + \tilde{u}_j = f_{n_{k_j}}(x)  + \mu(x)|\nabla u_{n_{k_j}}|^2 + \Delta u - u, \quad \mbox{in } X.
\end{equation*}

Consider the test function $\tilde v_j = \widetilde{E}_j \tilde u_j$ where $\widetilde{E}_j = \exp(\tilde t \tilde u_j^2)$ and $\tilde t = 2 ||\mu||_{\infty}^2$. As $\tilde u_j \in X$ we have $\tilde v_j \in X$, and using the inequality
$$
|\nabla u_{n_{k_j}}|^2 \leq 2 (|\nabla \tilde u_j|^2 + |\nabla u|^2),
$$
we obtain 
\arraycolsep1.5pt
$$
\begin{array}{l}
\displaystyle
\int_{\Omega} \widetilde{E}_j(1+2 \tilde t \tilde u_j^2) |\nabla \tilde u_j|^2 \,dx 
+
\int_{\Omega}\widetilde{E}_j \tilde u_j^2 \,dx 
\\[3mm]
\hspace{15mm}
\displaystyle
= 
\int_{\Omega} \nabla \tilde u_j \nabla \tilde v_j \,dx 
+ 
\int_{\Omega} \tilde u_j \tilde v_j \,dx 
\\[3mm]
\hspace{15mm}
\displaystyle
=
\int_{\Omega}f_{n_{k_j}}(x) \tilde v_j \,dx 
+ 
\int_{\Omega} \mu(x) |\nabla u_{n_{k_j}}|^2 \tilde v_j \,dx 
\\[3mm]
\hfill\displaystyle
-
\int_{\Omega} \widetilde{E}_j \nabla u \nabla \tilde u_j (1+ 2 \tilde t \tilde u_j^2)\,dx 
- 
\int_{\Omega}u \tilde v_j \,dx
\\[3mm]
\hspace{15mm}
\displaystyle
\leq 
\int_{\Omega} f_{n_{k_j}}(x) \tilde E_j \tilde u_j\, dx 
-
\int_{\Omega} \tilde E_j \nabla u \nabla \tilde u_j (1+ 2 \tilde t \tilde u_j^2)\,dx 
- 
\int_{\Omega}u \tilde E_j \tilde u_j \,dx 
\\[3mm]
\displaystyle
\hfill
+ 
2 ||\mu||_{\infty} 
\left( \int_{\Omega}\tilde E_j^{1/2}|\tilde u_j| | \nabla \tilde u_j| | \nabla \tilde u_j| \tilde E_j^{1/2}\, dx  
+ \int_{\Omega}|\nabla u|^2 \tilde E_j \tilde u_j \,dx \right) 
\\[3mm]
\hspace{15mm}
\displaystyle
\leq 
\int_{\Omega} f_{n_{k_j}}(x) \tilde E_j \tilde u_j dx 
-
\int_{\Omega} \tilde E_j \nabla u \nabla \tilde u_j (1 + 2 \tilde t \tilde u_j^2) \,dx 
-
\int_{\Omega} u \tilde E_j \tilde u_j \,dx 
\\[3mm]
\hspace{30mm}
\displaystyle
+ 2 ||\mu||_{\infty}  
\left( ||\mu||_{\infty} \int_{\Omega} \tilde E_j |\tilde u_j|^2 |\nabla \tilde u_j|^2 \,dx 
\right.
\\[3mm]
\hfill\displaystyle
\left.
+ 
\frac{1}{4 ||\mu||_{\infty} } \int_{\Omega}\tilde E_j |\nabla \tilde u_j|^2 \,dx 
+ 
\int_{\Omega}|\nabla u|^2 \tilde E_j \tilde u_j \,dx \right) 
\\[3mm]
\displaystyle
\hspace{15mm}
\leq 
\int_{\Omega} f_{n_{k_j}}(x) \tilde E_j \tilde u_j\, dx 
-
\int_{\Omega}\tilde E_j \nabla u \nabla \tilde u_j (1 + 2 \tilde t \tilde u_j^2)\, dx -\int_{\Omega} u \tilde E_j \tilde u_j \,dx 
\\[3mm]
\hfill
\displaystyle
+ \frac{1}{2} \int_{\Omega}\tilde E_j (1+ 2 \tilde t \tilde u_j^2) |\nabla \tilde u_j|^2 dx 
+ 
2 ||\mu||_{\infty}  \int_{\Omega}|\nabla u|^2 \tilde E_j \tilde u_j\, dx.
\end{array}
$$
\arraycolsep5pt%
Hence we deduce that
\begin{equation}
\label{rap}
\begin{array}{l} 
\displaystyle
\frac{1}{2} \int_{\Omega} \tilde E_j(1+2 \tilde t \tilde u_j^2) |\nabla \tilde u_j|^2 dx  
  + 
  \int_{\Omega}\tilde E_j \tilde u_j^2 dx
\\[3mm]
\displaystyle
\hspace{15mm}
\leq 
\int_{\Omega} (f_{n_{k_j}}(x)-f(x)) \tilde E_j \tilde u_j dx  
+ 
2 ||\mu||_{\infty}  \int_{\Omega} |\nabla u|^2 \tilde E_j \tilde u_j dx 
\\[3mm] 
\displaystyle
\hspace{20mm}
- 
\int_{\Omega}\tilde E_j \nabla u \nabla \tilde u_j (1 + 2 \tilde t \tilde u_j^2) dx 
- 
\int_{\Omega} u \tilde E_j \tilde u_j dx
+
\int_{\Omega} f(x) \tilde E_j \tilde u_j dx.
\end{array}
\end{equation}
%

Let us prove that each of the terms on the right hand side converges to zero. For the first one, as the sequence $\{u_n\}$ is bounded in $ L^{\infty}(\Omega)$ there exists  $C_1 >0$ such that, for all $j \in \N$, $||\widetilde{E}_j||_{\infty} \leq C_1.$ This implies the existence of a constant $C>0$ such that
\begin{equation}
\label{21}
\lim_{j \to \infty} \left|\int_{\Omega} (f_{n_{k_j}}(x)-f(x)) \widetilde{E}_j \tilde u_j dx \right| \leq C \lim_{j \to \infty}||f_{n_{k_j}}-f||_p  = 0.
\end{equation}

For the second term we have $|\nabla u|^2 \widetilde{E}_j \tilde u_j \to 0$ a.e. in $\Omega$ as $\tilde u_j \to 0$ a.e. in $\Omega$ and $\widetilde{E}_j$ is bounded. Moreover, for all $j \in \N$,
$$
\left| |\nabla u|^2 \widetilde{E}_j \tilde u_j \right| \leq C C_1 |\nabla u|^2
$$
with $CC_1|\nabla u|^2 \in L^1(\Omega)$. Hence by Lebesgue's dominated convergence theorem we have that
$$
\int_{\Omega} |\nabla u|^2 \widetilde{E}_j \tilde u_j dx \to 0.
$$

To prove that  the third term converges to zero, observe that $\nabla \tilde u_j \rightharpoonup 0$ weakly in $L^2(\Omega)$. Hence if we prove that $\widetilde{E}_j \nabla u (1+ 2 \tilde t \tilde u_j^2) $ converges strongly in $L^2(\Omega)$, we shall obtain
$$\int_{\Omega} \widetilde{E}_j \nabla u \nabla \tilde u_j (1+ 2 \tilde t \tilde u_j^2) dx \to 0.$$

Observe that $\widetilde{E}_j \nabla u (1+ 2 \tilde t \tilde u_j^2) \to \nabla u$ a.e.  in $\Omega$. Moreover we have
$$
\left| \widetilde{E}_j \nabla u (1+ 2 \tilde t \tilde u_j^2) \right| \leq C_1 (1+ 2 \tilde t C^2) |\nabla u| 
\quad \mbox{ with } \quad C_1(1+ 2 \tilde t C^2) \nabla u \in L^2(\Omega).
$$ 
Hence, again by Lebesgue dominated convergence theorem, we have $\widetilde{E}_j \nabla u (1+ 2 \tilde t \tilde u_j^2) \to \nabla u$ strongly in $L^2(\Omega)$. 
For the two last terms observe that
$$
u \widetilde{E}_j \tilde u_j \to 0 \mbox{ a.e. in } \Omega \quad \mbox{ and } \quad  |u \widetilde{E}_j \tilde u_j| \leq CC_1 |u|
$$
with $CC_1 |u| \in L^1(\Omega)$. This holds true also for $f\widetilde{E}_j \tilde u_j$.
Hence again we have 
$$
\int_{\Omega} u \widetilde{E}_j \tilde u_j dx \to 0 \quad\mbox{ and }\quad \int_{\Omega} f \widetilde{E}_j \tilde u_j dx \to 0.
$$
This  implies, by (\ref{rap}), that
$$
\lim_{j \to \infty} ||\tilde u_j||^2 \leq \lim_{j \to \infty} 2 \left( \frac{1}{2} \int_{\Omega} \widetilde{E}_j(1+ 2 \tilde t \tilde u_j^2) |\nabla \tilde u_j|^2 dx + \int_{\Omega} \widetilde{E}_j \tilde u_j^2 dx \right) = 0.
$$
As $\tilde u_j \to 0$ weakly in $H$ we obtain $\tilde u_j \to 0$ strongly in $H$, namely $u_{n_{k_j}} \to u$ strongly in $H$. Hence we can pass to the limit in the equation and  $u \in X$ satisfies
\begin{equation*}
-\Delta u + u - \mu(x) |\nabla u|^2 = f, \quad \mbox{ in } \Omega,
\end{equation*}
At this point we have proved the continuity of $K^{\mu}$ from $L^p(\Omega)$ to $H$. \medskip

\noindent {\bf Step 4.} {\it $K^{\mu}$ is continuous from $L^p(\Omega)$ to $C(\overline\Omega)$.} \medskip 

Let $\{f_n\}\subset  L^p(\Omega)$ be a sequence such that $f_n\to f$ in $L^p(\Omega)$. In particular the sequence $\{f_n\}$ is bounded in $L^p(\Omega)$. Hence, by Step 1, for every subsequence $\{f_{n_k}\}_k$ the set $\{u_{n_k}=K^{\mu}(f_{n_k})\mid k\in \mathbb N\}$ is relatively compact in $C(\overline\Omega)$ i.e. there exists a subsequence $(u_{n_{k_j}})_j$ which converges in $C(\overline\Omega)$ to $v\in C(\overline\Omega)$. By Step 3, $u_{n_{k_j}}\to u = K^{\mu}(f)$ in $H$. In particular 
$u_{n_{k_j}}\to v$ in $C(\overline\Omega)$ and $u_{n_{k_j}}\to u$ in $L^2(\Omega)$. By unicity of the limit, we conclude that $u=v$. As this is true for every subsequence, we have also that, if $f_n\to f$ in $L^p(\Omega)$ then $u_n = K^{\mu}(f_n)\to u = K^{\mu}(f)$ in $C(\overline\Omega)$ which concludes the proof.
\end{proof}



\begin{thebibliography}{99}




\bibitem{AbBi2}
{\sc H.\ Abdel Hamid, M.F.\ Bidaut-V\'{e}ron}, On the connection between two quasilinear elliptic problems 
with source terms of order 0 or 1,
{\em Comm. Contemp. Math.}, {\bf 12},  (2010), 727-788.


\bibitem{AbDaPe}
{\sc B.\ Abdellaoui, A.\ Dall'Aglio, I.\ Peral}, Some remarks on
elliptic problems with critical growth in the gradient, {\em J.
Differential Equations}, {\bf 222}, (2006), 21-62 + Corr. {\em J.
Differential Equations}, {\bf 246} (2009), 2988-2990.



\bibitem{AlPi}
{\sc N.\ Alaa, M.\ Pierre}, Weak solutions of some quasilinear elliptic equations with data measures,
{\em  SIAM J. Math. Anal.}, {\bf 24},  (1993), 23-35.


\bibitem{AlLiTr}
{\sc A.\ Alvino, P.L.\ Lions, G.\ Trombetti}, Comparison results for elliptic and parabolic equations via Schwarz symmetrization, {\em  Ann. Inst. H. Poincar\'e, Analyse non lin\'eaire}, {\bf 7}, (1990), 37-65.

\bibitem{AmCr} {\sc H.\ Amann, M.G.\ Crandall},
On some existence theorems for semi-linear elliptic equations,
{\em Indiana Univ. Math. J.}, {\bf 27}, (1978),  779--790.

\bibitem{AmAr}
{\sc A.\ Ambrosetti, D.\ Arcoya}, An introduction to nonlinear functional analysis and elliptic problems, {Birkhauser}, 2011.



\bibitem{ArCaLeMAOrPe}
{\sc D.\ Arcoya, J.\ Carmona, T.\ Leonori, P. J.\  Mart\'{\i}nez-Aparicio, L.\ Orsina, F.\ Petitta},
Existence and nonexistence of solutions for singular quadratic quasilinear equations,
{\em J. Differential Equations}, {\bf 246}, (2009) 4006-4042.


\bibitem{ArSe} {\sc D.\ Arcoya, S.\ Segura de Le\'on}, Uniqueness of solutions for some elliptic equations with a quadratic gradient term,  {\em ESAIM: Control, Optimisation and Calculus of Variations}, {\bf 16}, (2010), 327-336.

\bibitem{BaBlGeKo}
{\sc G.\ Barles, A.P.\ Blanc, C.\ Georgelin,  M.\ Kobylanski},
Remarks on the maximum principle for nonlinear elliptic PDE with
quadratic growth conditions, {\em Ann. Scuola Norm. Sup. Pisa},
{\bf 28}, (1999), 381-404.

\bibitem{BaMu}
{\sc G.\ Barles, F.\ Murat}, Uniqueness and the maximum principle
for quasilinear elliptic equations with quadratic growth
conditions, {\em Arch. Rational. Mech Anal.}, {\bf 133}, (1995),
77-101.

\bibitem{BaPo}
{\sc G.\ Barles, A.\ Porretta}, Uniqueness for unbounded solutions
to stationary viscous Hamilton-Jacobi equations,  {\em Ann. Sc.
Norm. Super. Pisa Cl. Sci.}, {\bf 5}, (2006),  107-136.


\bibitem{BeBoMu} {\sc A.\ Bensoussan, L.\ Boccardo, F.\ Murat},  On a non linear partial differential equation having 
natural growth terms and unbounded solution, {\em Ann. Inst. H. Poincar\'e}   {\bf 5}, (1988), 347--364. 

\bibitem{BeFr}
{\sc A.\ Bensoussan, J.\ Frehse},  Nonlinear elliptic systems in stochastic game theory. {\em J. Reine Ungew. Math.}, {\bf 350}, (1984), 23-67.

\bibitem{BeMeMuPo} {\sc F.\ Betta, A.\ Mercaldo, F.\ Murat, M.\ Porzio},  Uniqueness results for nonlinear elliptic equations with a lower order term, {\em Nonlinear Anal. TMA}, {\bf 63}, (2005), 153--170. 

\bibitem{BoGaMu} {\sc L.\ Boccardo, T. Gallou\"et, F.\ Murat}, A unified presentation of two existence results for problems with natural growth, {\em in Progress in partial differential equations: the Metz surveys}, 2 (1992),  127--137, {\em Pitman Res. Notes Math. Ser.}, 296, {\em Longman Sci. Tech., Harlow}, 1993.


\bibitem{BoMuPu1}
{\sc L.\ Boccardo, F.\ Murat, J.P.\ Puel}, Existence de solutions faibles pour des \'{e}quations elliptiques quasi-lin\'{e}aires \`{a} croissance quadratique, {\em Nonlinear partial differential equations and their applications. Coll\`{e}ge de France Seminar, Vol. IV (Paris, 1981/1982)},  19--73, {\em Res. Notes in Math.}, 84, {\em Pitman, Boston, Mass.-London}, 1983.

\bibitem{BoMuPu1.5}
{\sc L.\ Boccardo, F.\ Murat, J.P.\ Puel},
Existence de solutions non born\' ees pour certaines \' equations quasi-lin\'eaires, {\em Portugaliae Mathematica}, {\bf 41}, (1982), 507--534.
 

\bibitem{BoMuPu2}
{\sc L.\ Boccardo, F.\ Murat, J.P.\ Puel}, R\'{e}sultats d'existence pour certains
probl\`{e}mes elliptiques quasilin\'{e}aires, {\em Ann. Scuola Norm. Sup. Pisa}, {\bf
11}, (1984), 213--235.

\bibitem{BoMuPu2.5} {\sc L.\ Boccardo, F.\ Murat, J.P.\ Puel}, Existence of bounded solutions for nonlinear elliptic unilateral problems, {\it Ann. Mat. Pura Appl.}, {\bf 152}, (1988), 183--196.

\bibitem{BoMuPu4}{\sc L.\ Boccardo, F.\ Murat, J.P.\ Puel}, Quelques propri\'et\'es des op\'erateurs
elliptiques quasi-lin\'eaires,  {\em  C. R. Acad. Sci. Paris
S\'{e}r. I Math.}, {\bf 307},  (1988),   749--752.

\bibitem{BoMuPu3}
{\sc L.\ Boccardo, F.\ Murat, J.P.\ Puel}, $L^\infty$ estimate for some
nonlinear elliptic partial differential equations and application to an
existence result, {\em  SIAM J. Math. Anal.}, {\bf 23},  (1992),   326--333.

\bibitem{BoMa} {\sc G.\  Bottaro, M.E.\ Marina}, Problema di Dirichlet per equazioni ellitiche di tipo variazionale su insiemi non limitati, {\it Boll. Un. Mat. Ital.},  {\bf 8}, (1973), 46--56.


\bibitem{BT77} {\sc H.\ Brezis, R.E.L.\ Turner},
 On a class of superlinear elliptic problems,
{\em Comm. Partial Differ. Equations}, {\bf 2}, (1977), 601-614.

\bibitem{DoGi}
{\sc P.\ Donato, D.\ Giachetti}, Quasilinear elliptic equations with quadratic growth in unbounded domains. 
 {\em Nonlinear Anal. TMA, }, {\bf 10}, (1986), 791--804.








\bibitem{FeMu1}
{\sc V.\ Ferone, F.\ Murat}, Quasilinear problems having quadratic growth in
the gradient: an existence result when the source term is small, {\em \'{E}quations
aux d\'{e}riv\'{e}es partielles et applications, Gauthier-Villars, \'{E}d. Sci. M\'{e}d.
Elsevier, Paris,}  (1998), 497--515.

\bibitem{FeMu2}
{\sc V.\ Ferone, F.\ Murat}, Nonlinear problems having quadratic growth in the
gradient: an existence result when the source term is small, {\em Nonlinear Anal.
TMA}, {\bf 42}, (2000), 1309--1326.

\bibitem{FePo}
{\sc V.\ Ferone, M.R.\ Posteraro}, On a class of quasilinear elliptic equations with quadratic growth in the gradient, {\em Nonlinear Anal.
TMA}, {\bf 20}, (1993), 703--711.

\bibitem{FePoRa}
{\sc V.\ Ferone, M.R.\ Posteraro, J.M.\ Rakotoson,} $L^{\infty}$-estimates for nonlinear elliptic problems with $p$-growth in the gradient, {\em J. Ineq. Appl.}, {\bf 2}, (1999), 109--125.

\bibitem{Fr}
{\sc J.\ Frehse}, On the regularity of solutions to elliptic differential inequalities, {\em Mathematical Techniques of Optimisation, Control and Decision (Annals of Ceremade)}, 
(Edited by J.P. Aubin, A. Bensoussan and I. Ekeland), Birkhauser, Boston, (1981).

%

\bibitem{Gi} {\sc M.\ Giaquinta}, { Multiple integrals in the calculus of variations and nonlinear elliptic systems.}, Annals of Mathematics Studies, 105. Princeton University Press, Princeton, NJ, 1983.

\bibitem{GiMo} {\sc  M.\ Giaquinta, G.\ Modica}, Regularity results for some classes of higher order nonlinear elliptic systems, {\em J. Reine Angew. Math.}, {\bf  311/312},  (1979),   145--169.

\bibitem{GiTu} {\sc D.\ Gilbarg, N.S.\ Trudinger}, Elliptic partial differential equations of
second order, 2nd ed., Springer, 1983.

\bibitem{GrMo} {\sc  N.\ Grenon-Isselkou, J.\ Mossino}, Existence de solutions born\'ees pour certaines \'equations elliptiques quasilin\'eaires, {\em C. R. Math. Acad.
Sci. Paris}, {\bf  321},  (1995),   51--56.




\bibitem{JeSi}
{\sc L.\ Jeanjean, B.\ Sirakov}, Existence and multiplicity for elliptic problems with quadratic growth in the gradient, {\em Comm. Part. Diff. Equ.}, {\bf 38}, (2013), 244--264.

\bibitem{KaKr}
{\sc J.L.\ Kazdan, R.J.\ Kramer}, Invariant criteria for existence of solutions
to second-order quasilinear elliptic equations, {\em Comm. Pure Appl. Math.},
{\bf 31}, (1978), 619--645.

\bibitem{LU68}
{\sc O.\ Ladyzenskaya, N.\ Ural'tseva}, Linear and Quasilinear Elliptic Equations, {translated by Scripta Technica, Academic Press,
New York}, 1968. 

\bibitem{LeSc}
{\sc J.\ Leray, J.\ Schauder}, Topologie et \'equations fonctionnelles, {\em Ann. Sci. Ecole Norm. Sup.}, 
{\bf 51}, (1934), 45-78.


\bibitem{MaPaSa} 
{\sc C.\ Maderna, C.\ Pagani, S.\ Salsa}, Quasilinear elliptic equations
with quadratic growth in the gradient, {\em J. Differential Equations}, {\bf
97}, (1992), 54--70.

\bibitem{Mi}
{\sc C.\ Miranda}, Alcuni teoremi di inclusione, {\em Ann. Polon. Math.}, {\bf 16}, (1965), 305-315.

\bibitem{Po}
{\sc A. Porretta}, The {\it ergodic limit} for a viscous Hamilton Jacobi equation with Dirichlet conditions, {\em Rend. Lincei Mat. Appl.}, {\bf 21}, (2010), 59-78.

\bibitem{Ra}
{\sc P.H.\ Rabinowitz}, A global theorem for nonlinear eigenvalue problems and applications, {\it Contributions to nonlinear functional analysis (Proc. Sympos. Math. Res. Center, Univ Wisconsin, Madison, Wis}, {\em Academic Press, New York} (1971),  11-36.

\bibitem{Si}
{\sc B.\ Sirakov}, Solvability of uniformly elliptic fully nonlinear PDE, {\em
Arch. Rat. Mech. Anal.},  {\bf 195}, (2010), 579--607.



\bibitem{Tr}
{\sc N.S.\ Trudinger}, Linear elliptic operators with measurable coefficients,
{\em Annali Scuola Norm. Sup. Pisa, Cl. Scienze, 3e s\'erie},  {\bf 27}, (1973), 265--308.

\bibitem{WeYe}
{\sc J.\ Wei, D.\ Ye}, On MEMS equation with Fringing field, {\em Proc. AMS.}, {\bf 138}, (2010), 1693-1699.




\end{thebibliography}
\end{document}